\documentclass{amsart}
\usepackage{amsfonts,amssymb,latexsym}
\usepackage{amsmath,amsthm}
\usepackage{eucal}
\usepackage[latin1]{inputenc}

\newtheorem{theorem}{Theorem}[section]
\newtheorem{lemma}[theorem]{Lemma}
\newtheorem{proposition}[theorem]{Proposition}
\newtheorem{corollary}[theorem]{Corollary}

\theoremstyle{definition}
\newtheorem{definition}[theorem]{Definition}

\newtheorem*{merci}{Acknowledgements}

\theoremstyle{remark}
\newtheorem{remark}{Remark}[section]
\def\R{{\mathbb R}}

\numberwithin{equation}{section}

\begin{document}

\title[Bilinear Strichartz estimates for the ZK equation and applications ]{Bilinear Strichartz estimates for the  Zakharov-Kuznetsov equation and applications}
\author[L. Molinet and D. Pilod]{Luc Molinet $^{\star}$ and Didier Pilod $^{\dagger}$}

\subjclass[2000]{Primary ; Secondary } \keywords{Zakharov-Kuznetsov equation, Initial value
problem, Bilinear Strichartz estimates, Bourgain's spaces}

\maketitle

\vspace{-0.5cm}

{\scriptsize \centerline{$^{\star}$ LMPT, Universit\' e Fran\c cois Rabelais
Tours, CNRS UMR 7350, F\'ed\'eration Denis Poisson-CNRS,}
             \centerline{ Parc Grandmont, 37200 Tours, France.}
             \centerline{email: Luc.Molinet@lmpt.univ-tours.fr}

\vspace{0.1cm}
             \centerline{$^{\dagger}$  Instituto de Matem\'atica,
             Universidade Federal do Rio de Janeiro,}
              \centerline{Caixa Postal 68530, CEP: 21945-970, Rio
              de Janeiro, RJ, Brazil.}
              \centerline{email: didier@im.ufrj.br}}

\vspace{0.5cm}

\begin{abstract} This article is concerned with the Zakharov-Kuznetsov equation
\begin{equation}  \label{ZK0}
\partial_tu+\partial_x\Delta u+u\partial_xu=0  .
\end{equation} 
We prove that the associated initial value problem is locally well-posed in $H^s(\mathbb R^2)$ for $s>\frac12$ and globally well-posed in $H^1(\mathbb R\times \mathbb T)$ and in $H^s(\R^3) $ for $ s>1$. Our main new ingredient is a bilinear Strichartz estimate in the context of Bourgain's spaces which  allows to control the high-low frequency interactions appearing in the nonlinearity of \eqref{ZK0}. In the $\mathbb R^2$ case, we also need to use a recent result by Carbery, Kenig and Ziesler on sharp Strichartz estimates for homogeneous dispersive operators. Finally,  to prove the global well-posedness result  in $ \R^3 $, we need to use the atomic spaces introduced by Koch and Tataru.
\end{abstract}

\section{Introduction}

The Zakharov-Kuznetsov equation (ZK)
\begin{equation} \label{ZK}
\partial_tu+\partial_x\Delta u+u\partial_xu=0,
\end{equation}
where $u=u(x,y,t)$ is a real-valued function, $t \in \mathbb R$, $x
\in \mathbb R$, $y \in \mathbb R$, $\mathbb T$ or $\mathbb R^2$ and $\Delta$ is the
laplacian, was introduced by Zakharov and Kuznetsov in \cite{ZK} to
describe the propagation of ionic-acoustic waves in magnetized
plasma.  The derivation of ZK from the Euler-Poisson system
with magnetic field was performed by Lannes, Linares and Saut \cite{LLS} (see also \cite{LS} for a formal derivation). Moreover, the
following quantities are conserved by the flow of ZK,
\begin{equation} \label{M}
M(u)=\int u(x,y,t)^2dxdy,
\end{equation}
and
\begin{equation} \label{H}
H(u)=\frac12 \int \big( |\nabla u(x,y,t)|^2-\frac13u(x,y,t)^3\big)dxdy.
\end{equation}
Therefore $L^2$ and $H^1$ are two natural spaces to study the well-posedness
for the ZK equation.

In the 2D case, Faminskii proved in \cite{Fa} that the Cauchy
problem associated to \eqref{ZK} was well-posed in the energy space
$H^1(\mathbb R^2)$. This result was recently improved by Linares and
Pastor who proved well-posedness in $H^s(\mathbb R^2)$, for $s >
3/4$.  Both results were proved by using a fixed point argument taking 
advantage of the dispersive smoothing effects associated to the linear part 
of ZK, following the ideas of Kenig, Ponce and Vega \cite{KPV} for the KdV equation.

The case of the cylinder $\mathbb R \times \mathbb T$ was
treated by Linares, Pastor and Saut in \cite{LPS}. They obtained
well-posedness in $H^s(\mathbb R \times \mathbb T)$ for $s>\frac32$.
Note that the best results in the 3D case were obtained last year by
Ribaud and Vento \cite{RV} (see also Linares and Saut \cite{LS} for former results). They proved
local well-posedness in $H^s(\mathbb R^3)$ for $s>1$ and in $B_2^{1,1}(\mathbb R^3)$. However that it is
still an open problem to obtain global solutions in $\mathbb R \times \mathbb T$ and $\mathbb R^3$.

The objective of this article is to improve the local well-posedness
results for the ZK equation in $\mathbb R^2$ and $\mathbb R \times \mathbb T$, and to  prove new global well-posedness results. In this direction, we obtain the global well-posedness in $H^1(\mathbb R \times \mathbb T)$ and in $ H^s(\R^3) $ for $ s>1$. Next are our main results.

\begin{theorem} \label{theoR2}
Assume that $s>\frac12$. For any $u_0 \in H^s(\mathbb R^2)$, there exists $T=T(\|u_0\|_{H^s})>0$ and a unique solution of \eqref{ZK} such that $u(\cdot,0)=u_0$ and 
\begin{equation} \label{theoR2.1}
u \in C([0,T]:H^s(\mathbb R^2)) \cap X_T^{s,\frac12+} \ .
\end{equation}
Moreover, for any $T' \in (0,T)$, there exists a neighborhood $\mathcal{U}$ of $u_0$ in $H^s(\mathbb R^2)$, such that the flow map data-solution 
\begin{equation} \label{theoR2.2}
S: v_0 \in \mathcal{U} \mapsto v \in  C([0,T']:H^s(\mathbb R^2)) \cap X_{T'}^{s,\frac12+}
\end{equation}
is smooth.
\end{theorem}

\begin{theorem} \label{theoRT}
Assume that $s \ge 1$. For any $u_0 \in H^s(\mathbb R \times \mathbb T)$, there exists $T=T(\|u_0\|_{H^s})>0$ and a unique solution of \eqref{ZK} such that $u(\cdot,0)=u_0$ and 
\begin{equation} \label{theoRT.1}
u \in C([0,T]:H^s(\mathbb R \times \mathbb T)) \cap X_T^{s,\frac12+} \ .
\end{equation}
Moreover, for any $T' \in (0,T)$, there exists a neighborhood $\widetilde{\mathcal{U}}$ of $u_0$ in $H^s(\mathbb R \times \mathbb T)$, such that the flow map data-solution 
\begin{equation} \label{theoRT.2}
S: v_0 \in \mathcal{U} \mapsto v \in  C([0,T']:H^s(\mathbb R \times \mathbb T)) \cap X_{T'}^{s,\frac12+}
\end{equation}
is smooth.
\end{theorem}

\begin{remark} The spaces $X_T^{s,b}$ are defined in Section \ref{notation}
\end{remark}

As a consequence of Theorem \ref{theoRT}, we deduce the following result by using the conserved quantities $M$ and $H$ defined in \eqref{M} and \eqref{H}.
\begin{theorem} \label{theoRTglobal}
The initial value problem associated to the Zakharov-Kuznetsov equation is globally well-posed in $H^1(\mathbb R \times \mathbb T)$.
\end{theorem}

\begin{remark}
Theorem \ref{theoRTglobal} provides a good setting to apply the techniques of  Rousset  and Tzvetkov \cite{RT1}, \cite{RT2} and prove the transverse instability of the KdV soliton for the ZK equation.
\end{remark}
Finally, we combine the conserved quantities $M$ and $H$ with a well-posedness result in the Besov space $ B^{1,1}_2 $ and interpolation arguments  to prove :
\begin{theorem} \label{theo3}
The initial value problem associated to the Zakharov-Kuznetsov equation is globally well-posed in $H^s(\mathbb R^3 )$ for any $ s>1$.
\end{theorem}

\begin{remark} Note that the  global well-posedness for the ZK equation in the energy space $H^1(\mathbb R^3)$ is still  an open problem.
\end{remark}

The main new ingredient in the proofs of Theorems \ref{theoR2},  \ref{theoRT} and \ref{theo3} is a bilinear estimate in the context of Bourgain's spaces (see for instance  the work of Molinet, Saut and Tzvetkov for the the KPII equation \cite{MST} for similar estimates), which allows to control the interactions between high and low frequencies appearing in the nonlinearity of \eqref{ZK}. In the $\mathbb R^2$ case, we also need to use a recent result by Carbery, Kenig and Ziesler on sharp Strichartz estimates for homogeneous dispersive operators. This allows us to treat the case of high-high to high frequency interactions. With those estimates in hand, we are able to derive the crucial bilinear estimates (see Propositions \ref{BilinR2} and \ref{BilinRT} below) and conclude the proof of Theorems \ref{theoR2} and \ref{theoRT} by using a fixed point argument in Bourgain's spaces. To prove the global wellposedness in $ \R^3 $ we follows ideas in \cite{Bou} and need to get a suitable lower bound on the time before the norm of solution doubles.  To get this bound we will have to work in the framework of the atomic spaces $ U^2_S $ and $V^2_S $ introduced by Koch and Tataru in \cite{KT}.

We saw very recently on the arXiv that Gr\"unrock and Herr obtained a similar result \cite{GH} in the $\mathbb R^2$ case by using the same kind of techniques. Note however that they do not need to use the Strichartz estimate derived by Carbery, Kenig and Ziesler. On the other hand, they use a linear transformation on the equation to obtain a symmetric symbol $\xi^3+\eta^3$ in order to apply their arguments. Since we derive our bilinear estimate directly on the original equation,  our method of proof also worked in the $\mathbb R \times \mathbb T$ setting (see the results in Theorems \ref{theoRT} and \ref{theoRTglobal}). 

This paper is organized as follows: in the next section we introduce the notations and define the function spaces. In Section 3, we recall the linear Strichartz estimates for ZK and derive our crucial bilinear estimate. Those estimates are used in Section 4 and 5 to prove the bilinear estimates in $\mathbb R^2$ and $\mathbb R \times \mathbb T$. Finally, Section 6 is devoted to the $ \R^3 $ case.

\section{Notation, function spaces and linear estimates} \label{notation}

\subsection{Notation}
For any positive numbers $a$ and $b$, the notation $a \lesssim b$
means that there exists a positive constant $c$ such that $a \le c
b$. We also write $a \sim b$ when $a \lesssim b$ and $b \lesssim a$.
If $\alpha \in \mathbb R$, then $\alpha_+$, respectively $\alpha_-$,
will denote a number slightly greater, respectively lesser, than
$\alpha$. If $A$ and $B$ are two positive numbers, we use the
notation $A\wedge B=\min(A,B)$ and $A \vee B=\max(A,B)$. Finally,
$\text{mes} \, S$ or $|S|$ denotes the Lebesgue measure of a
measurable set $S$ of $\mathbb R^n$, whereas $\# F$ or $|S|$ denotes
the cardinal of a finite set $F$.

We use the notation $|(x,y)|=\sqrt{3x^2+y^2}$ for  $(x,y) \in \mathbb R^2$.
For $u=u(x,y,t) \in \mathcal{S}(\mathbb R^3)$, $\mathcal{F}(u)$, or
$\widehat{u}$, will denote its space-time Fourier transform, whereas
$\mathcal{F}_{xy}(u)$, or $(u)^{\wedge_{xy}}$, respectively
$\mathcal{F}_t(u)=(u)^{\wedge_t}$, will denote its Fourier transform
in space, respectively in time. For $s \in \mathbb R$, we define the
Bessel and Riesz potentials of order $-s$, $J^s$ and $D^s$, by
\begin{displaymath}
J^su=\mathcal{F}^{-1}_{xy}\big((1+|(\xi,\mu)|^2)^{\frac{s}{2}}
\mathcal{F}_{xy}(u)\big) \quad \text{and} \quad
D^su=\mathcal{F}^{-1}_{xy}\big(|(\xi,\mu)|^s \mathcal{F}_{xy}(u)\big).
\end{displaymath}

Throughout the paper, we fix a smooth cutoff function $\eta$ such that
\begin{displaymath}
\eta \in C_0^{\infty}(\mathbb R), \quad 0 \le \eta \le 1, \quad
\eta_{|_{[-5/4,5/4]}}=1 \quad \mbox{and} \quad  \mbox{supp}(\eta)
\subset [-8/5,8/5].
\end{displaymath}
For $k \in \mathbb N^{\star}=\mathbb Z \cap [1,+\infty)$, we define
\begin{displaymath}
\phi(\xi)=\eta(\xi)-\eta(2\xi), \quad
\phi_{2^k}(\xi,\mu):=\phi(2^{-k}|(\xi,\mu)|).
\end{displaymath}
and
\begin{displaymath}
\psi_{2^k}(\xi,\mu,\tau)=\phi(2^k(\tau-(\xi^3+\xi\mu^2))).
\end{displaymath}
By convention, we also denote
\begin{displaymath}
\phi_1(\xi,\mu)=\eta(|(\xi,\mu)|), \quad \text{and} \quad
\psi_1(\xi,\mu,\tau)=\eta(\tau-(\xi^3+\xi\mu^2)).
\end{displaymath}
Any summations over capitalized variables such as $N, \, L$, $K$ or
$M$ are presumed to be dyadic with $N, \, L$, $K$ or $M \ge 1$,
\textit{i.e.}, these variables range over numbers of the form $\{2^k
: k \in \mathbb N\} $. Then, we have that
\begin{displaymath}
\sum_{N}\phi_N(\xi,\mu)=1, \quad \mbox{supp} \, (\phi_N) \subset
\{\frac{5}{8}N\le |(\xi,\mu)| \le \frac{8}5N\}=:I_N \ , \ N \ge 1,
\end{displaymath}
and
\begin{displaymath}
\mbox{supp} \, (\phi_1) \subset \{|(\xi,\mu)| \le \frac85\}=:I_1  .
\end{displaymath}
Let us define the Littlewood-Paley multipliers by
\begin{equation}\label{proj}
P_Nu=\mathcal{F}^{-1}_{xy}\big(\phi_N\mathcal{F}_{xy}(u)\big), \quad
Q_Lu=\mathcal{F}^{-1}\big(\psi_L\mathcal{F}(u)\big).
\end{equation}

Finally, we denote by $e^{-t\partial_x\Delta}$ the free group
associated with the linearized part of equation \eqref{ZK}, which is
to say,
\begin{equation} \label{V}
\mathcal{F}_{xy}\big(e^{-t\partial_x\Delta}\varphi
\big)(\xi,\mu)=e^{itw(\xi,\mu)}\mathcal{F}_{xy}(\varphi)(\xi,\mu),
\end{equation}
where $w(\xi,\mu)=\xi^3+\xi\mu^2$. We also define the resonance
function $H$ by
\begin{equation} \label{Resonance}
H(\xi_1,\mu_1,\xi_2,\mu_2)=w(\xi_1+\xi_2,\mu_1+\mu_2)-w(\xi_1,\mu_1)-w(\xi_2,\mu_2).
\end{equation}
Straightforward computations give that
\begin{equation} \label{Resonance2}
H(\xi_1,\mu_1,\xi_2,\mu_2)=3\xi_1\xi_2(\xi_1+\xi_2)+\xi_2\mu_1^2+\xi_1\mu_2^2+2(\xi_1+\xi_2)\mu_1\mu_2.
\end{equation}

We make the obvious modifications when working with $u=u(x,y)$ for
$(x,y) \in \mathbb R \times \mathbb T $ and denote by $q$ the Fourier variable corresponding to $y$.

\subsection{Function spaces}

For $1 \le p \le \infty$, $L^p(\mathbb R^2)$ is the usual Lebesgue
space with the norm $\|\cdot\|_{L^p}$, and for $s \in \mathbb R$ ,
the real-valued Sobolev space $H^s(\mathbb R^2)$  denotes the space
of all real-valued functions with the usual norm
$\|u\|_{H^s}=\|J^su\|_{L^2}.$
If $u=u(x,y,t)$ is a function defined for $(x,y) \in
\mathbb R^2$ and $t$ in the time interval $[0,T]$, with $T>0$, if $B$
is one of the spaces defined above, $1 \le p \le \infty$ and $1 \le q \le \infty$, we will
define the mixed space-time spaces $L^p_TB_{xy}$,
$L^p_tB_{xy}$, $L^q_{xy}L^p_T$ by the norms
\begin{displaymath}
\|u\|_{L^p_TB_{xy}} =\Big(
\int_0^T\|u(\cdot,\cdot,t)\|_{B}^pdt\Big)^{\frac1p} \quad , \quad
\|u\|_{L^p_tB_{xy}} =\Big( \int_{\mathbb R}\|u(\cdot,\cdot,t)\|_{B}^pdt\Big)^{\frac1p},
\end{displaymath}
and
\begin{displaymath}
\|u\|_{L^q_{xy}L^p_T}= \left(\int_{\mathbb R^2}\Big( \int_0^T|u(x,y,t)|^pdt\Big)^{\frac{q}{p}}dx\right)^{\frac1q},
\end{displaymath}
if $1 \le p, \ q < \infty$ with the obvious modifications in the case $p=+\infty$ or $q=+\infty$.

For $s$, $b \in \mathbb R$, we introduce the Bourgain spaces
$X^{s,b}$ related to the linear part of \eqref{ZK} as
the completion of the Schwartz space $\mathcal{S}(\mathbb R^3)$
under the norm
\begin{equation} \label{Bourgain}
\|u\|_{X^{s,b}} = \left(
\int_{\mathbb{R}^3}\langle\tau-w(\xi,\mu)\rangle^{2b}\langle
|(\xi,\mu)|\rangle^{2s}|\widehat{u}(\xi,\mu, \tau)|^2 d\xi d\mu
d\tau \right)^{\frac12},
\end{equation}
where $\langle x\rangle:=1+|x|$. Moreover, we define a localized (in time) version of these spaces.
Let $T>0$ be a positive time. Then, if $u: \mathbb R^2 \times
[0,T]\rightarrow \mathbb C$, we have that
\begin{displaymath}
\|u\|_{X^{s,b}_{T}}=\inf \{\|\tilde{u}\|_{X^{s,b}} \ : \ \tilde{u}:
\mathbb R^2 \times \mathbb R \rightarrow \mathbb C, \
\tilde{u}|_{\mathbb R^2 \times [0,T]} = u\}.
\end{displaymath}

We make the obvious modifications for functions defined on $(x,y,t) \in \mathbb R
\times \mathbb Z \times \mathbb R$. In particular, the integration over $\mu \in \mathbb R$
in \eqref{Bourgain} is replaced by a summation over $q \in \mathbb Z$, which is to say
\begin{equation} \label{Bourgainper}
\|u\|_{X^{s,b}} = \left( \sum_{q \in \mathbb Z}\int_{\mathbb
R^2}\langle\tau-w(\xi,q)\rangle^{2b}\langle
|(\xi,q)|\rangle^{2s}|\widehat{u}(\xi,q, \tau)|^2 d\xi  d\tau
\right)^{\frac12},
\end{equation}
where $w(\xi,q)=\xi^3+\xi q^2$.

\subsection{Linear estimates in the $X^{s,b}$ spaces.}
In this subsection, we recall some well-known estimates for
Bourgain's spaces (see \cite{Gi} for instance).

\begin{lemma}[Homogeneous linear estimate] \label{prop1.1}
Let $s \in \mathbb R$ and $b>\frac12$. Then
\begin{equation} \label{prop1.1.2}
\|\eta(t) e^{-t\partial_x\Delta} f\|_{X^{s,b}} \lesssim\|f\|_{H^s} \ .
\end{equation}
\end{lemma}

\begin{lemma}[Non-homogeneous linear estimate] \label{prop1.2}
Let $s \in \mathbb R$. Then for any $0<\delta<\frac12$,
\begin{equation} \label{prop1.2.1}
\big\|\eta(t)\int_0^t e^{-(t-t')\partial_x\Delta}g(t')dt'\big\|_{X^{s,\frac12+\delta}} \lesssim  \|g\|_{X^{s,-\frac12+\delta}} \ .
\end{equation}
\end{lemma}

\begin{lemma} \label{prop1.3b}
For any $T>0$, $s \in \mathbb R$ and for all $-\frac12< b' \le b
<\frac12$, it holds
\begin{equation} \label{prop1.3b.1}
\|u\|_{X^{s,b'}_T} \lesssim T^{b-b'}\|u\|_{X^{s,b}_T}.
\end{equation}
\end{lemma}

\section{Linear and bilinear Strichartz estimates}

\subsection{Linear strichartz estimates on $\mathbb R^2$}
First, we state a Strichartz estimate for the unitary group $\{e^{-t\partial_x\Delta}\}$ proved by Linares and Pastor (\textit{c.f.} Proposition 2.3 in \cite{LP}). 
\begin{proposition} \label{Strichartz} 
Let $0 \le \epsilon < \frac12$ and $0 \le \theta \le 1$. Assume that $(q,p)$ satisfy $p=\frac{2}{1-\theta}$ and $q=\frac{6}{\theta(2+\epsilon)}$. Then, we have that 
\begin{equation} \label{Strichartz1}
\|D_x^{\frac{\theta\epsilon}{2}}e^{-t\partial_x\Delta}\varphi\|_{L^q_tL^p_{xy}} \lesssim \|\varphi\|_{L^2}
\end{equation}
for all $\varphi \in L^2(\mathbb R^2)$.
\end{proposition}
Then, we obtain the following corollary in the context of Bourgain' spaces.
\begin{corollary} \label{Strichartzcoro}
We have that
\begin{equation} \label{Strichartzcoro1} 
\|u\|_{L^4_{xyt}} \lesssim \|u\|_{X^{0,\frac56+}}, 
\end{equation}
for all $u \in X^{0,\frac56+}$.
\end{corollary}
\begin{proof} Estimate \eqref{Strichartz1} in the case $\epsilon=0$ and $\theta=\frac35$ writes 
\begin{equation} \label{Strichartzcoro2}
\|e^{-t\partial_x\Delta}\varphi\|_{L^5_{xyt}} \lesssim \|\varphi\|_{L^2}
\end{equation}
for all $\varphi \in L^2(\mathbb R^2)$. A classical argument (see for example \cite{Gi}) yields 
\begin{displaymath}
\|u\|_{L^5_{xyt}} \lesssim \|u\|_{X^{0,\frac12+}},
\end{displaymath}
which implies estimate \eqref{Strichartzcoro1} after interpolation with Plancherel's identity $\|u\|_{L^2_{xyt}}=\|u\|_{X^{0,0}}$.
\end{proof}

In \cite{CKZ}, Carbery, Kenig and Ziesler proved an optimal $L^4$-restriction theorem for
homogeneous polynomial hypersurfaces in $\mathbb R^3$.
\begin{theorem} \label{CKZ}
Let $\Gamma(\xi,\mu)=(\xi,\mu,\Omega(\xi,\mu))$, where
$\Omega(\xi,\mu)$ is a polynomial, homogeneous of degree $d \ge 2$.
Then there exists a positive constant $C$ (depending on $\phi$) such
that
\begin{equation} \label{CKZ1}
\Big(\int_{\mathbb R^2}|\widehat{f}(\Gamma(\xi,\mu))|^2
|K_{\Omega}(\xi,\mu) |^{\frac14} d\xi d\mu\Big)^{\frac12} \le C
\|f\|_{L^{4/3}},
\end{equation}
for all $f \in L^{4/3}(\mathbb R^3)$ and where
\begin{equation} \label{CKZ2}
|K_{\Omega}(\xi,\mu) |= \big| \det \text{Hess} \, \Omega(\xi,\mu)
\big|.
\end{equation}
\end{theorem}

As a consequence, we have the following corollary.
\begin{corollary} \label{CKZcoro}
Let $|K_{\Omega}(D)|^{\frac18}$ and $e^{it\Omega(D)}$ be the Fourier
multipliers associated to $|K_{\Omega}(\xi,\mu)|^{\frac18}$ and
$e^{it\Omega(\xi,\mu)}$, \textit{i.e.}
\begin{equation} \label{CKZcoro1}
\mathcal{F}_{xy}\Big(|K_{\Omega}(D)|^{\frac18}\varphi
\Big)(\xi,\mu)= |K_{\Omega}(\xi,\mu)
|^{\frac18}\mathcal{F}_{xy}(\varphi)(\xi,\mu)
\end{equation}
where $\big|K_{\Omega}(\xi,\mu) \big|$ is defined in \eqref{CKZ2},
and
\begin{equation} \label{CKZcoro2}
\mathcal{F}_{xy}\big(e^{it\Omega(D)}\varphi
\big)(\xi,\mu)=e^{it\Omega(\xi,\mu)}\mathcal{F}_{xy}(\varphi)(\xi,\mu).
\end{equation}
Then,
\begin{equation} \label{CKZcoro3}
\big\|
|K_{\Omega}(D)|^{\frac18}e^{it\Omega(D)}\varphi\big\|_{L^4_{xyt}}
\lesssim \|\varphi\|_{L^2},
\end{equation}
for all $\varphi \in L^2(\mathbb R^2)$.
\end{corollary}

\begin{proof} By duality, it suffices to prove that
\begin{equation} \label{CKZcoro4}
\int_{\mathbb
R^3}|K_{\Omega}(D)|^{\frac18}e^{it\Omega(D)}\varphi(x,y)
f(x,y,t)dxdydt \lesssim \|\varphi\|_{L^2_{xy}}\|f\|_{L^{4/3}_{xyt}}.
\end{equation}
The Cauchy-Schwarz inequality implies that it is enough to prove that
\begin{equation} \label{CKZcoro5}
\Big\| \int_{\mathbb R}|K_{\Omega}(D)
|^{\frac18}e^{-it\Omega(D)}f(x,y,t)dt\Big\|_{L^2_{xy}} \lesssim
\|f\|_{L^{4/3}_{xyt}}
\end{equation}
in order to prove estimate \eqref{CKZcoro4}. But straightforward computations give
\begin{displaymath}
\mathcal{F}_{x,y}\Big(\int_{\mathbb R}\big|K_{\Omega}(D)
\big|^{\frac18}e^{-it\Omega(D)}fdt \Big)(\xi,\mu)
=c|K_{\Omega}(\xi,\mu)
\big|^{\frac18}\mathcal{F}_{x,y,t}(f)(\xi,\mu,\Omega(\xi,\mu)),
\end{displaymath}
so that estimate \eqref{CKZcoro5} follows directly from Plancherel's identity and estimate \eqref{CKZ1}.
\end{proof}

Now, we apply Corollary \ref{CKZcoro} in the case of the unitary group $e^{-t\partial_x\Delta }$.
\begin{proposition} \label{Strichartzlin}
Let $|K(D)|^{\frac18}$ be the Fourier multiplier associated to
$|K(\xi,\mu) |^{\frac18}$ where
\begin{equation} \label{Strichartzlin1}
|K(\xi,\mu) |=|3\xi^2-\mu^2 |
\end{equation}
Then, we have that
\begin{equation} \label{Strichartzlin2}
\big\|
|K(D)|^{\frac18}e^{-t\partial_x\Delta}\varphi\big\|_{L^4_{xyt}}
\lesssim \|\varphi\|_{L^2}
\end{equation}
for all $\varphi \in L^2(\mathbb R^2)$, and
\begin{equation} \label{Strichartzlin3}
\big\| |K(D)|^{\frac18}u\big\|_{L^4_{xyt}} \lesssim
\|u\|_{X^{0,\frac12+}}
\end{equation}
for all $u \in X^{0,\frac12+}$.
\end{proposition}

\begin{proof} The symbol associated to $e^{-t\partial_x\Delta}$ is
given by $w(\xi,\mu)=\xi^3+\xi\mu^2$. After an easy computation, we
get that
\begin{displaymath}
\det \text{Hess} \, w(\xi,\mu)=4(3\xi^2-\mu^2).
\end{displaymath}
Estimate \eqref{Strichartzlin2} follows then as a direct application
of Corollary \ref{CKZcoro}.

\end{proof}

\begin{remark} It follows by applying estimate \eqref{Strichartz1} with $\epsilon=1/2-$ and $\theta=2/3+$ that
\begin{displaymath}
\|D_x^{\frac16}e^{-t\partial_x\Delta}\varphi\|_{L^{6-}_{xyt}} \lesssim
\|\varphi\|_{L^2},
\end{displaymath}
for all $\varphi \in L^2(\mathbb R^2)$, which implies in the context
of Bourgain's spaces (after interpolating with the trivial estimate
$\|u\|_{L^2_{xyt}}=\|u\|_{X^{0,0}}$) that
\begin{equation} \label{Strichartzlinremark}
\|D_x^{\frac18}u\|_{L^4_{xyt}} \lesssim \|u\|_{X^{0,\frac38+}},
\end{equation}
for all $u \in X^{0,\frac38+}$.

Estimate \eqref{Strichartzlin3} can be viewed as an improvement of
estimate \eqref{Strichartzlinremark}, since outside of the lines
$|\xi|=\frac1{\sqrt{3}}|\mu|$, it allows to recover $1/4$ of
derivatives instead of $1/8$ of derivatives in $L^4$.
\end{remark}

\begin{remark}
it is interesting to observe that the resonance function $H$ defined
in \eqref{Resonance2} cancels out on the planes
$(\xi_1=-\frac{\mu_1}{\sqrt{3}},\xi_2=\frac{\mu_2}{\sqrt{3}})$ and
$(\xi_1=\frac{\mu_1}{\sqrt{3}},\xi_2=-\frac{\mu_2}{\sqrt{3}})$.
\end{remark}

\subsection{Bilinear Strichartz estimates}

In this subsection, we prove the following crucial bilinear
estimates related to the ZK dispersion relation for functions
defined on $\mathbb R^3$ and $\mathbb R \times \mathbb T \times
\mathbb R$.

\begin{proposition} \label{BilinStrichartzI}
Let $N_1, \ N_2, \ L_1, \ L_2$ be dyadic numbers in $\{ 2^k  :  k \in \mathbb N^{\star} \} \cup \{1\}$.
Assume that $u_1$ and $u_2$ are two functions
in $L^2(\mathbb R^3)$ or $L^2(\mathbb R \times \mathbb T \times \mathbb R)$. Then,
\begin{equation} \label{BilinStrichartzI0}
\begin{split}
\|(P_{N_1}Q_{L_1}u_1)&(P_{N_2}Q_{L_2}u_2)\|_{L^2} \\ & \lesssim
(L_1 \wedge L_2)^{\frac12}(N_1 \wedge N_2)\|P_{N_1}Q_{L_1}u_1\|_{L^2}
\|P_{N_2}Q_{L_2}u_2\|_{L^2}
\end{split}
\end{equation}
Assume moreover that $N_2 \ge
4N_1$ or $N_1 \ge 4N_2$. Then,
\begin{equation} \label{BilinStrichartzI1}
\begin{split}
\|(P_{N_1}&Q_{L_1}u_1) (P_{N_2}Q_{L_2}u_2)\|_{L^2} \\ & \lesssim
\frac{(N_1 \wedge N_2)^{\frac12}}{N_1 \vee N_2}(L_1 \vee
L_2)^{\frac12}(L_1 \wedge L_2)^{\frac12}\|P_{N_1}Q_{L_1}u_1\|_{L^2}
\|P_{N_2}Q_{L_2}u_2\|_{L^2}.
\end{split}
\end{equation}
\end{proposition}

\begin{remark} Estimate \eqref{BilinStrichartzI1}  will be very useful to control the
high-low frequency interactions in the nonlinear term of \eqref{ZK}.
\end{remark}

In the proof of Proposition \ref{BilinStrichartzI}  we will need some basic Lemmas stated in
\cite{MST}.

\begin{lemma} \label{basicI}
Consider a set $\Lambda \subset \mathbb R \times X$, where
$X=\mathbb R$ or $\mathbb T$. Let the projection on the $\mu$ axis
be contained in a set $I \subset \mathbb R$. Assume in addition that
there exists $C>0$ such that for any fixed $\mu_0 \in I \cap X$, $|\Lambda
\cap \{(\xi,\mu_0) \ : \ \mu_0 \in X\}| \le C$. Then, we get that 
$|\Lambda| \le C |I|$ in the case where $X=\mathbb R$ and $|\Lambda| \le C( |I|+1)$ 
in the case where $X=\mathbb T$.
\end{lemma}

The second one is a direct consequence of the mean value theorem.
\begin{lemma} \label{basicII}
Let $I$ and $J$ be two intervals on the real line and $f:J
\rightarrow \mathbb R$ be a smooth function. Then,
\begin{equation} \label{basicII1}
\text{mes} \, \{x \in J \ : \ f(x) \in I\} \le \frac{|I|}{\inf_{\xi
\in J}|f'(\xi)|}.
\end{equation}
\end{lemma}

In the case where $f$ is a polynomial of degree $3$, we also have the following result.
\begin{lemma} \label{basicIII}
Let $a \neq 0, \ b, \ c$ be real numbers and $I$ be an interval on
the real line. Then,
\begin{equation} \label{basicIII1}
\text{mes} \, \{x \in J \ : \ ax^2+bx+c \in I\} \lesssim
\frac{|I|^{\frac12}}{|a|^{\frac12}}.
\end{equation}
and
\begin{equation} \label{basicIII2}
\# \{q \in \mathbb Z  \ : \ aq^2+bq+c \in I\} \le
\frac{|I|^{\frac12}}{|a|^{\frac12}}+1.
\end{equation}
\end{lemma}

\begin{proof}[Proof of Proposition \ref{BilinStrichartzI}]
We prove estimates \eqref{BilinStrichartzI0}--\eqref{BilinStrichartzI1} in the case where $(x,y,t) \in \mathbb R^3$. The case $(x,y,t) \in \mathbb R \times \mathbb T \times \mathbb R$ follows in a similar way. The
Cauchy-Schwarz inequality and Plancherel's identity yield
\begin{equation} \label{BilinStrichartzI3}
\begin{split}
\|(P_{N_1}Q_{L_1}u_1)(P_{N_2}&Q_{L_2}u_2)\|_{L^2} \\ &=
\|(P_{N_1}Q_{L_1}u_1)^{\wedge} \star (P_{N_2}Q_{L_2}u_2)^{\wedge}\|_{L^2} \\ & \lesssim
\sup_{(\xi,\mu,\tau) \in \mathbb R^3}|A_{\xi\,\mu,\tau} |^{\frac12}
\|P_{N_1}Q_{L_1}u_1\|_{L^2}\|P_{N_2}Q_{L_2}u_2\|_{L^2},
\end{split}
\end{equation}
where
\begin{displaymath}
\begin{split}
A_{\xi,\mu,\tau}=&\Big\{  (\xi_1,\mu_1,\tau_1) \in \mathbb R^3 \ : \
 |(\xi_1,\mu_1)| \in I_{N_1}, \
|(\xi-\xi_1,\mu-\mu_1)| \in I_{N_2} \\ & \quad \quad
|\tau_1-w(\xi_1,\mu_1)| \in I_{L_1}, \
|\tau-\tau_1-w(\xi-\xi_1,\mu-\mu_1)| \in I_{L_2}\Big\} \ .
\end{split}
\end{displaymath}
it remains then to estimate the measure of the set
$A_{\xi,\mu,\tau}$ uniformly in $(\xi,\mu,\tau) \in \mathbb R^3$. 

To obtain \eqref{BilinStrichartzI0}, we use the trivial estimate
\begin{displaymath} 
|A_{\xi\,\mu,\tau} | \lesssim (L_1 \wedge L_2)(N_1 \wedge N_2)^2, 
\end{displaymath} 
for all $(\xi,\mu,\tau) \in \mathbb R^3$.

Now we turn to the proof of estimate \eqref{BilinStrichartzI1}. First, we get easily from the triangle inequality that 
\begin{equation} \label{BilinStrichartzI4}
|A_{\xi\,\mu,\tau} | \lesssim (L_1 \wedge L_2) |B_{\xi\,\mu,\tau} |,
\end{equation}
where 
\begin{equation}  \label{BilinStrichartzI40}
\begin{split}
B_{\xi,\mu,\tau}=&\Big\{  (\xi_1,\mu_1) \in \mathbb R^2 \ : \
 |(\xi_1,\mu_1)| \in I_{N_1}, \
|(\xi-\xi_1,\mu-\mu_1)| \in I_{N_2} \\ & \quad \quad
|\tau-w(\xi,\mu) -H(\xi_1,\xi-\xi_1,\mu_1,\mu-\mu_1)| \lesssim L_1 \vee L_2\Big\} 
\end{split}
\end{equation}
and $H(\xi_1,\xi_2,\mu_1,\mu_2)$ is the resonance function defined in \eqref{Resonance2}. Next, we observe from the hypotheses on the daydic numbers $N_1$ and $N_2$ that 
\begin{displaymath} 
\Big| \frac{\partial H}{\partial \xi_1}(\xi_1,\xi-\xi_1,\mu_1,\mu-\mu_1)\Big|=
\big|3\xi_1^2+\mu_1^2-(3(\xi-\xi_1)^2+(\mu-\mu_1)^2) \big| \gtrsim (N_1 \vee N_2)^2 \  .
\end{displaymath}
Then, if we define
$B_{\xi,\mu,\tau}(\mu_1) =\{ \xi_1 \in \mathbb R \ : \ (\xi_1,\mu_1) \in B_{\xi,\mu,\tau}\}$, we deduce applying estimate \eqref{basicII1} that 
\begin{displaymath}  
|B_{\xi,\mu,\tau}(\mu_1) | \lesssim \frac{L_1 \vee L_2}{(N_1 \vee N_2)^2} \ ,
\end{displaymath}
for all $\mu_1 \in \mathbb R$. Thus, it follows from Lemma \ref{basicI} that 
\begin{equation} \label{BilinStrichartzI5} 
|B_{\xi,\mu,\tau} | \lesssim \frac{N_1 \wedge N_2}{(N_1 \vee N_2)^2} (L_1 \wedge L_2 )\ .
\end{equation}
Finally, we conclude the proof of estimate \eqref{BilinStrichartzI1} gathering estimates \eqref{BilinStrichartzI3}--\eqref{BilinStrichartzI5}. 
\end{proof}

\section{Bilinear estimate in $\mathbb R \times \mathbb R$}

The main result of this section is stated below. 
\begin{proposition} \label{BilinR2}
Let $s>\frac12$. Then, there exists $\delta>0$ such that 
\begin{equation} \label{BilinR2.1}
\|\partial_x(uv)\|_{X^{s,-\frac12+2\delta}} \lesssim \|u\|_{X^{s,\frac12+\delta}}\|v\|_{X^{s,\frac12+\delta}},
\end{equation}
for all $u, \ v: \mathbb R^3 \rightarrow \mathbb R$ such that $u, \ v \in X^{s,\frac12+\delta}$.
\end{proposition}

Before proving Proposition \ref{BilinR2}, we give a technical lemma. 
\begin{lemma} \label{technicalR2}
Assume that $0<\alpha<1$. Then, we have that 
\begin{equation} \label{technicalR2.1}
\begin{split}
|(\xi_1+\xi_2,&\mu_1+\mu_2)|^2 \\ &\le \big||(\xi_1,\mu_1)|^2-|(\xi_2,\mu_2)|^2 \big|+f(\alpha)\max\big\{|(\xi_1,\mu_1)|^2,  |(\xi_2,\mu_2)|^2 \big\},
\end{split}
\end{equation}
for all $(\xi_1,\mu_1), \ (\xi_2,\mu_2) \in \mathbb R^2$ satisfying 
\begin{equation} \label{technicalR2.2}
(1-\alpha)^{\frac12}\sqrt{3}|\xi_i| \le |\mu_i| \le (1-\alpha)^{-\frac12}\sqrt{3}|\xi_i|, \quad \text{for} \ i=1,2,
\end{equation}
and 
\begin{equation} \label{technicalR2.3}
\xi_1\xi_2<0 \quad \text{and} \quad \mu_1\mu_2<0,
\end{equation}
and where $f$ is a continuous function on $[0,1]$ satisfying $\lim_{\alpha \to 0}f(\alpha)=0$.
We also recall te notation $|(\xi,\mu)|=\sqrt{3\xi^2+\mu^2}$.
\end{lemma} 
 
\begin{proof} If we denote by $\vec{u}_1=(\xi_1,\mu_1)$, $\vec{u}_2=(\xi_2,\mu_2)$ and $(\vec{u}_1,\vec{u}_2)_{e}=3\xi_1\xi_2+\mu_1\mu_2$ the scalar product associated to $|\cdot|$, then \eqref{technicalR2.1} is equivalent to 
\begin{equation} \label{technicalR2.4}
|\vec{u}_1+\vec{u}_2|^2\le \big||\vec{u}_1|^2-|\vec{u}_2|^2 \big|+f(\alpha)\max\big\{|\vec{u}_1|^2,  |\vec{u}_2|^2 \big\}.
\end{equation}
Moreover, without loss of generality, we can always assume that 
\begin{equation} \label{technicalR2.5}
\xi_1>0, \ \mu_1>0, \ \xi_2<0, \ \mu_2<0 \ \text{and}  \ |\vec{u}_1| \ge |\vec{u}_2|.
\end{equation} 
Thus, it suffices to prove that 
\begin{equation} \label{technicalR2.6} 
(\vec{u}_1+\vec{u}_2,\vec{u}_2)_e \le \frac{f(\alpha)}2|\vec{u}_1|^2.
\end{equation}
By using \eqref{technicalR2.2} and \eqref{technicalR2.3}, we have that 
\begin{equation} \label{technicalR2.7}
\begin{split}
(\vec{u}_1+\vec{u}_2,\vec{u}_2)_e  &=3(\xi_1+\xi_2)\xi_2+(\mu_1+\mu_2)\mu_2 \\ & 
\le 6(\xi_1+\xi_2)\xi_2-3\alpha\xi_1\xi_2+3\big((1-\alpha)^{-1}-1 \big)\xi_2^2
\end{split}
\end{equation}
On the other hand, the assumptions $\xi_1>0$, $\xi_2<0$, $|\vec{u}_1| \ge |\vec{u}_2|$ and \eqref{technicalR2.2} imply that
\begin{equation} \label{technicalR2.8}
\xi_1=|\xi_1| \ge (1-g(\alpha))|\xi_2|=-(1-g(\alpha))\xi_2
\end{equation}
with 
\begin{displaymath}
g(\alpha)=1-\Big(\frac{2-\alpha}{2+3\big((1-\alpha)^{-1}-1 \big)}\Big)^{\frac12} \underset{\alpha \to 0}{\longrightarrow} 0.
\end{displaymath}
Thus, it follows gathering \eqref{technicalR2.7} and \eqref{technicalR2.8} that 
\begin{displaymath} 
(\vec{u}_1+\vec{u}_2,\vec{u}_2)_e \le 6g(\alpha)\xi_2^2-3\alpha\xi_1\xi_2+3\big((1-\alpha)^{-1}-1 \big)\xi_2^2,
\end{displaymath}
which implies \eqref{technicalR2.6} by choosing 
\begin{displaymath} 
f(\alpha)=12g(\alpha)+6\alpha+6\big((1-\alpha)^{-1}-1 \big)\underset{\alpha \to 0}{\longrightarrow} 0.
\end{displaymath}
\end{proof}

\begin{proof}[Proof of Proposition \ref{BilinR2}]By duality, it suffices to prove that 
\begin{equation} \label{BilinR2.2}
I \lesssim \|u\|_{L^2_{x,y,t}}\|v\|_{L^2_{x,y,t}}\|w\|_{L^2_{x,y,t}},
\end{equation}
where 
\begin{displaymath} 
I=\int_{\mathbb R^6}\Gamma^{\xi_1,\mu_1,\tau_1}_{\xi,\mu,\tau}\widehat{w}(\xi,\mu,\tau)\widehat{u}(\xi_1,\mu_1,\tau_1)\widehat{v}(\xi_2,\mu_2,\tau_2) d\nu,
\end{displaymath}
$\widehat{u}$, $\widehat{v}$ and $\widehat{w}$ are nonnegative functions, and we used the following notations
\begin{equation} \label{BilinR2.20}
\begin{split}
&\Gamma^{\xi_1,\mu_1,\tau_1}_{\xi,\mu,\tau}=|\xi| \langle |(\xi,\mu)| \rangle^s\langle \sigma\rangle^{-\frac12+2\delta}\langle |(\xi_1,\mu_1)| \rangle^{-s}\langle \sigma_1 \rangle^{-\frac12-\delta}
\langle |(\xi_2,\mu_2)| \rangle^{-s}\langle \sigma_2 \rangle^{-\frac12-\delta}, \\
&d\nu=d\xi d\xi_1d\mu d\mu_1d\tau d\tau_1, \quad \xi_2=\xi-\xi_1, \ \mu_2=\mu-\mu_1, \ \tau_2=\tau-\tau_1, \\ 
&\sigma=\tau-w(\xi,\mu) \quad \text{and} \quad \sigma_i=\tau_i-w(\xi_i,\mu_i), \ i=1,2.
\end{split}
\end{equation}

By using dyadic decompositions on the spatial frequencies of $u$, $v$ and $w$, we rewrite $I$ as 
\begin{equation} \label{BilinR2.3}
I=\sum_{N_1,N_2,N}I_{N,N_1,N_2} ,
\end{equation}
where 
\begin{displaymath} 
I_{N,N_1,N_2}=\int_{\mathbb R^6}\Gamma^{\xi_1,\mu_1,\tau_1}_{\xi,\mu,\tau}\widehat{P_Nw}(\xi,\mu,\tau)\widehat{P_{N_1}u}(\xi_1,\mu_1,\tau_1)\widehat{P_{N_2}v}(\xi_2,\mu_2,\tau_2) d\nu .
\end{displaymath}
Since $(\xi,\mu)=(\xi_1,\mu_1)+(\xi_2,\mu_2)$, we can split the sum into the following cases: 
\begin{itemize}
\item[(1)] $Low \times Low \to Low$ interactions: $N_1\le 2, N_2 \le 2, N \le 2$. In this case, we denote 
$$I_{LL \to L}=\sum_{N \le 4, N_1\le 4, N_2 \le 4}I_{N,N_1,N_2}.$$
\item[(2)] $Low \times High \to High$ interactions: $4 \le N_2,  N_1 \le N_2/4$  ($\Rightarrow N_2/2 \le N \le 2N_2$). In this case, we denote 
$$I_{LH \to H}=\sum_{4 \le N_2, N_1 \le N_2/4, N_2/2 \le N \le 2N_2}I_{N,N_1,N_2}.$$
\item[(3)] $High \times Low \to High$ interactions: $4 \le N_1, N_2 \le N_1/4$ ($\Rightarrow N_1/2 \le N \le 2N_1$). In this case, we denote 
$$I_{HL\to H}=\sum_{4 \le N_1, N_2 \le N_1/4, N_1/2 \le N \le 2N_1}I_{N,N_1,N_2}.$$
\item[(4)] $High \times High \to Low$ interactions: $4 \le N_1,N \le N_1/4$ ($\Rightarrow N_1/2 \le N_2 \le 2N_1$) 
or $4 \le N_2, N \le N_2/4$ ($\Rightarrow N_2/2 \le N_1 \le 2N_2$) .
In this case, we denote 
$$I_{HH \to L}=\sum_{4 \le N_1, N \le N_1/4, N_2/2 \le N_1 \le 2N_2}I_{N,N_1,N_2}.$$
\item[(5)] $High \times High \to High$ interactions: $N_2 \ge 4$, $N_1 \ge 4$, $N_2/2 \le N_1 \le 2N_2$, $N_1/2 \le N \le 2N_1$ and $N_2/2 \le N \le 2N_2$. In this case, we denote 
$$I_{HH \to H}=\sum_{N_2/2 \le N_1 \le 2N_2, N_1/2 \le N \le 2N_1, N_2/2 \le N \le 2N_2}I_{N,N_1,N_2}.$$
\end{itemize} 
Then, we have 
\begin{equation} \label{BilinR2.4}
I=I_{LL \to L}+I_{LH \to H}+I_{HL \to H}+I_{HH \to L}+I_{HH \to H}.
\end{equation}

\noindent \textit{1. Estimate for $I_{LL \to L}$.} We observe from Plancherel's identity, H\"older's inequality and estimate \eqref{Strichartzcoro1} that
\begin{equation} \label{BilinR2.40}
\begin{split}
I_{N,N_1,N_2} &\lesssim \big\|\Big(\frac{\widehat{P_{N_1}u}}{\langle\sigma_1\rangle^{\frac12+\delta}}\Big)^{\vee}\big\|_{L^4} \big\|\Big(\frac{\widehat{P_{N_2}v}}{\langle\sigma_2\rangle^{\frac12+\delta}}\Big)^{\vee}\big\|_{L^4}\|P_Nw\|_{L^2} \\ & \lesssim \|P_{N_1}u\|_{L^2}\|P_{N_2}v\|_{L^2}\|P_Nw\|_{L^2},
\end{split}
\end{equation}
which yields 
\begin{equation} \label{BilinR2.400}
I_{LL \to L} \lesssim \|u\|_{L^2}\|v\|_{L^2}\|w\|_{L^2}.
\end{equation}

\noindent \textit{2. Estimate for $I_{LH \to H}$.} In this case, we also use dyadic decompositions on the modulations variables $\sigma$, $\sigma_1$ and $\sigma_2$, so that 
\begin{equation} \label{BilinR2.5}
I_{N,N_1,N_2}=\sum_{L,L_1,L_2}I_{N,N_1,N_2}^{L,L_1,L_2},
\end{equation}
where 
\begin{displaymath} 
I_{N,N_1,N_2}^{L,L_1,L_2}=\int_{\mathbb R^6}\Gamma^{\xi_1,\mu_1,\tau_1}_{\xi,\mu,\tau}\widehat{P_NQ_Lw}(\xi,\mu,\tau)\widehat{P_{N_1}Q_{L_1}u}(\xi_1,\mu_1,\tau_1)\widehat{P_{N_2}Q_{L_2}v}(\xi_2,\mu_2,\tau_2) d\nu .
\end{displaymath}
Hence, by using the Cauchy-Schwarz inequality in $(\xi,\mu,\tau)$, we can bound $I_{N,N_1,N_2}^{L,L_1,L_2}$ by 
\begin{displaymath} \label{BilinR2.6}
N_2N_1^{-s}L^{-\frac12+2\delta}L_1^{-\frac12-\delta}L_2^{-\frac12-\delta}\|(P_{N_1}Q_{L_1}u) (P_{N_2}Q_{L_2}v)\|_{L^2}\|P_NQ_Lw\|_{L^2}.
\end{displaymath}
Now, estimate \eqref{BilinStrichartzI1}  provides the following bound for $I_{LH \to H}$,
\begin{displaymath} 
 \sum_{L,L_1,L_2}L^{-\frac12+2\delta}L_1^{-\delta}L_2^{-\delta} \sum_{N\sim N_2, N_1 \le N_2/4}N_1^{-(s-\frac12)}\|P_{N_1}Q_{L_1}u\|_{L^2}\|P_{N_2}Q_{L_2}v\|_{L^2}\|P_NQ_Lw\|_{L^2} .
\end{displaymath}
Therefore, we deduce after summing over $L, \ L_1, \ L_2, \ N_1$ and applying the Cauchy-Schwarz inequality in $N \sim N_2$ that
\begin{equation} \label{BilinR2.7}
\begin{split}
I_{LH \to H} &\lesssim \|u\|_{L^2}\sum_{N \sim N_2}\|P_{N_2}v\|_{L^2}\|P_Nw\|_{L^2} \\ &
\lesssim \|u\|_{L^2}\big(\sum_{N_2}\|P_{N_2}v\|_{L^2}^2\big)^{\frac12}\big(\sum_{N}\|P_{N}w\|_{L^2}^2\big)^{\frac12} \\ & \lesssim \|u\|_{L^2}\|v\|_{L^2}\|w\|_{L^2}.
\end{split}
\end{equation}

\noindent \textit{3. Estimate for $I_{HL \to H}$.} Arguing similarly, we get that
\begin{equation} \label{BilinR2.8}
I_{HL \to H} 
\lesssim \|u\|_{L^2}\|v\|_{L^2}\|w\|_{L^2}.
\end{equation}

\noindent  \textit{4. Estimate for $I_{HH \to L}$.} We use the same decomposition as in \eqref{BilinR2.5}. By using the Cauchy-Schwarz inequality, we can bound $I_{N,N_1,N_2}^{L,L_1,L_2}$ by
\begin{equation} \label{BilinR2.9}
 L^{-\frac12+2\delta}L_1^{-\frac12-\delta}L_2^{-\frac12-\delta} \frac{N^{s+1}}{N_1^sN_2^s}
\big\|(\widetilde{P_{N_1}Q_{L_1}u}) (P_{N}Q_{L}w)\big\|_{L^2} \|P_{N_2}Q_{L_2}v\|_{L^2},
\end{equation}
where $\tilde{f}(\xi,\mu,\tau)=f(-\xi,-\mu,-\tau)$.  Moreover, observe interpolating \eqref{BilinStrichartzI0} and \eqref{BilinStrichartzI1} that
\begin{equation} \label{BilinR2.10}
\begin{split}
&\|(\widetilde{P_{N_1}Q_{L_1}u}) (P_{N}Q_{L}w)\|_{L^2} \\ & \lesssim
\frac{(N_1 \wedge N)^{\frac12(1+\theta)}}{(N_1 \vee N)^{1-\theta}}(L_1 \vee
L)^{\frac12(1-\theta)}(L_1 \wedge L)^{\frac12}\|P_{N_1}Q_{L_1}u\|_{L^2}
\|P_{N}Q_{L}w\|_{L^2},
\end{split}
\end{equation}
for all $0 \le \theta \le 1$. Without loss of generality, we can assume that $L=L \vee L_1$ (the case $L_1=L \vee L_1$ is actually easier). Hence, we deduce from \eqref{BilinR2.9} and \eqref{BilinR2.10} that 
\begin{equation} \label{BilinR2.11}
 I_{N,N_1,N_2}^{L,L_1,L_2} \lesssim L_1^{-\delta}L_2^{-\frac12-\delta}L^{2\delta-\frac{\theta}2}N^{\frac12+\theta}N_1^{-(s-\theta)}\|P_{N_1}Q_{L_1}u\|_{L^2}\|P_{N}Q_{L}w\|_{L^2}\|P_{N_2}Q_{L_2}v\|_{L^2}
\end{equation}
Now, we choose $0<\theta<1$ and $\delta>0$ satisfying $0<2\theta<s-\frac12$ and $0<\delta<\frac{\theta}4$. It follows after summing \eqref{BilinR2.11} over $L, \ L_1, \ L_2$ and performing the Cauchy-Schwarz inequality in $N$ and $N_1$ that
\begin{equation} \label{BilinR2.12}
\begin{split}
I_{HH \to L} &\lesssim \sum_{N_1}N_1^{-(s-\frac12-2\theta)}\|P_{N_1}u\|_{L^2}\big(\sum_{N}\|P_{N}w\|_{L^2}^2\big)^{\frac12}\|v\|_{L^2} \\ &
\lesssim \|u\|_{L^2}\|w\|_{L^2}\|v\|_{L^2}.
\end{split}
\end{equation}

\noindent \textit{5. Estimate for $I_{HH \to H}$.} Let $0<\alpha <1$ be a small positive number such that $f(\alpha)=\frac{1}{1000}$, where $f$ is defined in Lemma \ref{technicalR2}. In order to simplify the notations, we will denote $(\xi,\mu,\tau)=(\xi_0,\mu_0,\tau_0)$. We split the integration domain in the following subsets: 
\begin{displaymath}
\begin{split}
\mathcal{D}_1&=\big\{ (\xi_1,\mu_1,\tau_1,\mu,\xi,\tau) \in \mathbb R^6 \ : \ (1-\alpha)^{\frac12}\sqrt{3}|\xi_i| \le |\mu_i| \le (1-\alpha)^{-\frac12}\sqrt{3}|\xi_i|, \ i=1,2 \big\},\\
\mathcal{D}_2&=\big\{ (\xi_1,\mu_1,\tau_1,\mu,\xi,\tau) \in \mathbb R^6 \ : \ (1-\alpha)^{\frac12}\sqrt{3}|\xi_i| \le |\mu_i| \le (1-\alpha)^{-\frac12}\sqrt{3}|\xi_i|, \ i=0,1 \big\},
\\ \mathcal{D}_3&=\big\{ (\xi_1,\mu_1,\tau_1,\mu,\xi,\tau) \in \mathbb R^6 \ : \ (1-\alpha)^{\frac12}\sqrt{3}|\xi_i| \le |\mu_i| \le (1-\alpha)^{-\frac12}\sqrt{3}|\xi_i|, \ i=0,2 \big\},
\\ \mathcal{D}_4&=\mathbb R^6 \setminus \bigcup_{j=1}^3\mathcal{D}_j \ .
\end{split}
\end{displaymath}
Then, if we denote by $I_{HH \to H}^j$ the restriction of $I_{HH \to H}$ to the domain $\mathcal{D}_j$, we have that 
\begin{equation} \label{BilinR2.13}
I_{HH \to H}=\sum_{j=1}^4I_{HH \to H}^j.
\end{equation}

\noindent  \textit{5.1. Estimate for $I^1_{HH \to H}$.} We consider the following subcases.
\begin{itemize} 
\item[(i)] \textit{Case $\big\{\xi_1\xi_2>0 \ \text{and}  \ \mu_1\mu_2>0\big\}$.} We define 
\begin{displaymath} 
\mathcal{D}_{1,1}=\big\{ (\xi_1,\mu_1,\tau_1,\mu,\xi,\tau) \in \mathcal{D}_1 \ : \ \xi_1\xi_2>0 \ \text{and}  \ \mu_1\mu_2>0\big\}
\end{displaymath}
and denote by $I_{HH \to H}^{1,1}$ the restriction of $I_{HH \to H}^1$ to the domain $\mathcal{D}_{1,1}$. We observe from \eqref{Resonance2} and the frequency localization that 
\begin{equation} \label{BilinR2.i.1}
\max\{|\sigma|,|\sigma_1|,|\sigma_2|\} \gtrsim  |H(\xi_1,\mu_1,\xi_2,\mu_2)| \gtrsim N_1^3
\end{equation}
in the region $\mathcal{D}_{1,1}$. Therefore, it follows arguing exactly as in \eqref{BilinR2.40} that 
\begin{equation} \label{BilinR2.i.2}
I_{HH \to H}^{1,1} \lesssim \|u\|_{L^2}\|v\|_{L^2}\|w\|_{L^2}.
\end{equation}

\item[(ii)] \textit{Case $\big\{\xi_1\xi_2>0 \ \text{and}  \ \mu_1\mu_2<0\big\}$ or $\big\{\xi_1\xi_2<0 \ \text{and}  \ \mu_1\mu_2>0\big\}$.} We define 
\begin{displaymath} 
\mathcal{D}_{1,2}=\big\{ (\xi_1,\mu_1,\tau_1,\mu,\xi,\tau) \in \mathcal{D}_1 \ : \ \xi_1\xi_2>0, \   \mu_1\mu_2<0\ \text{or} \ \xi_1\xi_2<0, \  \mu_1\mu_2>0\big\}
\end{displaymath}
and denote by $I_{HH \to H}^{1,2}$ the restriction of $I_{HH \to H}^1$ to the domain $\mathcal{D}_{1,2}$. Moreover, we use dyadic decompositions on the variables $\sigma$, $\sigma_1$ and $\sigma_2$ as in \eqref{BilinR2.5}. Plancherel's identity and the Cauchy-Schwarz inequality yield 
\begin{equation} \label{BilinR2.i.3}
I_{N,N_1,N_2}^{L,L_1,L_2} \lesssim N^{1-s}L^{-\frac12+2\delta}L_1^{-\frac12-\delta}L_2^{-\frac12-\delta}\|(P_{N_1}Q_{L_1}u) (P_{N_2}Q_{L_2}v)\|_{L^2}\|w\|_{L^2}.
\end{equation}
Next, we argue as in  \eqref{BilinStrichartzI3} to estimate $\|(P_{N_1}Q_{L_1}u) (P_{N_2}Q_{L_2}v)\|_{L^2}$. Moreover, we observe that
\begin{displaymath}
\Big|\frac{\partial H}{\partial \mu_1}(\xi_1,\xi-\xi_1,\mu_1,\mu-\mu_1) \Big|=2\big|\mu_1\xi_1-\mu_2\xi_2 \big| \gtrsim N^2
\end{displaymath}
in the region $\mathcal{D}_{1,2}$. Thus, we deduce from Lemma \ref{basicI}, estimates \eqref{basicII1} and \eqref{BilinStrichartzI3} and \eqref{BilinStrichartzI4} that
\begin{equation} \label{BilinR2.i.4}
\begin{split}
\|(P_{N_1}Q_{L_1}u)& (P_{N_2}Q_{L_2}v)\|_{L^2} \\ &\lesssim N^{-\frac12}(L_1 \vee L_2)^{\frac12}(L_1 \wedge L_2)^{\frac12}\|P_{N_1}Q_{L_1}u\|_{L^2}\|P_{N_2}Q_{L_2}v\|_{L^2}.
\end{split}
\end{equation}
Therefore, we deduce combining estimates \eqref{BilinR2.i.3} and \eqref{BilinR2.i.4} and summing over $L, \ L_1, \ L_2$ and $N \sim N_1 \sim N_2$ that
\begin{equation} \label{BilinR2.i.5}
I_{HH \to H}^{1,2} \lesssim \|u\|_{L^2}\|v\|_{L^2}\|w\|_{L^2}.
\end{equation}

\item[(iii)] \textit{Case $\big\{\xi_1\xi_2<0 \ \text{and}  \ \mu_1\mu_2<0\big\}$.} We define 
\begin{displaymath} 
\mathcal{D}_{1,3}=\big\{ (\xi_1,\mu_1,\tau_1,\mu,\xi,\tau) \in \mathcal{D}_1 \ : \ \xi_1\xi_2<0 \ \text{and}  \ \mu_1\mu_2<0\big\}
\end{displaymath}
and denote by $I_{HH \to H}^{1,3}$ the restriction of $I_{HH \to H}^1$ to the domain $\mathcal{D}_{1,3}$. Moreover, we observe  due to the frequency localization that there exists some $0<\gamma \ll 1$ such that
\begin{equation} \label{BilinR2.i.6}
\Big| |(\xi_2,\mu_2)|^2-|(\xi_1,\mu_1)|^2\Big| \ge\gamma \max\big\{ |(\xi_1,\mu_1)|^2, |(\xi_2,\mu_2)|^2\big\}
\end{equation}
in $\mathcal{D}_{1,3}$. Indeed, if estimate \eqref{BilinR2.i.6} does not hold for all $0<\gamma\le\frac1{1000}$, then estimate \eqref{technicalR2.1} with $f(\alpha)=\frac1{1000}$ would imply that  
\begin{displaymath} 
|(\xi,\mu)|^2 \le\frac1{500}\max\big\{ |(\xi_1,\mu_1)|^2, |(\xi_2,\mu_2)|^2\big\}
\end{displaymath}
which would be a contradiction since we are in the $High\times High \to High $ interactions case. Thus, we  deduce from \eqref{BilinR2.i.6} that 
\begin{displaymath} 
\Big|\frac{\partial H}{\partial \xi_1}(\xi_1,\xi-\xi_1,\mu_1,\mu-\mu_1) \Big| = \Big| |(\xi_2,\mu_2)|^2-|(\xi_1,\mu_1)|^2\Big| \gtrsim N^2.
\end{displaymath}
We can then reapply the arguments in the proof of Proposition \ref{BilinStrichartzI} to show that estimate \eqref{BilinR2.i.4} still holds true in this case. Therefore, we conclude arguing as above that 
\begin{equation} \label{BilinR2.i.7}
I_{HH \to H}^{1,3} \lesssim \|u\|_{L^2}\|v\|_{L^2}\|w\|_{L^2}.
\end{equation}
\end{itemize}
Finally, estimates \eqref{BilinR2.i.2}, \eqref{BilinR2.i.5} and \eqref{BilinR2.i.7} imply that 
\begin{equation} \label{BilinR2.i.8}
I_{HH \to H}^{1} \lesssim \|u\|_{L^2}\|v\|_{L^2}\|w\|_{L^2}.
\end{equation}

\noindent  \textit{5.2. Estimate for $I^2_{HH \to H}$ and $I^3_{HH \to H}$.} 
Arguing as for $I^1_{HH \to H}$, we get that 
\begin{equation} \label{BilinR2.ii.1}
I_{HH \to H}^{2}+ I_{HH \to H}^{3}\lesssim \|u\|_{L^2}\|v\|_{L^2}\|w\|_{L^2}.
\end{equation} 
We explain for example how to deal with $I_{HH \to H}^{2}$. It suffices to rewrite $I_{N,N_1,N_2}$ as 
\begin{displaymath} 
I_{N,N_1,N_2}=\int_{\mathcal{D}_2}\Gamma^{\tilde{\xi}_1,\tilde{\mu}_1,\tilde{\tau}_1}_{\xi,\mu,\tau}\widehat{P_Nw}(\xi,\mu,\tau)\widehat{\widetilde{P_{N_1}u}}(\tilde{\xi}_1,\tilde{\mu}_1,\tilde{\tau}_1)\widehat{P_{N_2}v}(\xi_2,\mu_2,\tau_2) d\tilde{\nu},
\end{displaymath}
where 
\begin{displaymath} 
d\tilde{\nu}=d\xi d\xi_2d\mu d\mu_2d\tau d\tau_2, \quad \tilde{\xi}_1=\xi_2-\xi, \ \tilde{\mu}_1=\mu_2-\mu, \ \tilde{\tau}_1=\tau_2-\tau,
\end{displaymath}
and $\Gamma^{\tilde{\xi}_1,\tilde{\mu}_1,\tilde{\tau}_1}_{\xi,\mu,\tau}$ is defined as in \eqref{BilinR2.20}. Moreover, we observe that 
\begin{displaymath}
\tilde{H}=H(\xi,\xi_2-\xi,\mu,\mu_2-\mu)=w(\xi_2,\mu_2)-w(\xi,\mu)-w(\xi_2-\xi,\mu_2-\mu)
\end{displaymath}
satisfies 
\begin{displaymath}
\Big|\frac{\partial \tilde{H}}{\partial \xi}\Big|=\big|3\xi^2+\mu^2-(3\tilde{\xi}_1^2+\tilde{\mu}_1^2) \big| \quad \text{and} \quad \Big|\frac{\partial \tilde{H}}{\partial \mu}\Big|=2\big|\xi\mu-\tilde{\xi}_1\tilde{\mu}_1 \big|.
\end{displaymath}
Therefore, we divide in the subregions $\big\{\xi\tilde{\xi}_1>0, \ \mu\tilde{\mu}_1>0 \}$, $\big\{\xi\tilde{\xi}_1<0, \ \mu\tilde{\mu}_1>0 \}$, $\big\{\xi\tilde{\xi}_1>0, \ \mu\tilde{\mu}_1<0 \}$ and $\big\{\xi\tilde{\xi}_1<0, \ \mu\tilde{\mu}_1<0 \}$ and use the same arguments as above. \\

\noindent  \textit{5.3. Estimate for $I^4_{HH \to H}$.} Observe that in the region $\mathcal{D}_4$, we have  
\begin{equation} \label{BilinR2.iv.1}
|\mu_i^2-3\xi_i^2|>\frac{\alpha}2 |(\xi_i,\mu_i)|^2 \quad \text{and} \quad |\mu_j^2-3\xi_j^2|>\frac{\alpha}2 |(\xi_j,\mu_j)|^2 ,
\end{equation} 
for at least a combination $(i,j)$ in $\{0,1,2\}$. Without loss of generality\footnote{in the other cases, we  cannot use estimate \eqref{Strichartzlin3} directly, but need to interpolate it with estimate \eqref{Strichartzcoro1} as previously.}, we can assume that $i=1$ and $j=2$ in \eqref{BilinR2.iv.1}. Then, we deduce from Plancherel's identity and H\"older's inequality that 
\begin{displaymath} 
I_{HH \to H}^4 \lesssim \sum_{  N_2 \sim N_1}N_1^{-(s-\frac12)}\|K(D)^{\frac18}\Big(\frac{\widehat{P_{N_1}u}}{\langle\sigma_1\rangle^{\frac12+\delta}}\Big)^{\vee}\|_{L^4} \|K(D)^{\frac18}\Big(\frac{\widehat{P_{N_2}v}}{\langle\sigma_2\rangle^{\frac12+\delta}}\Big)^{\vee}\|_{L^4}\|w\|_{L^2},
\end{displaymath}
where the operator $K(D)^{\frac18}$ is defined in Proposition \ref{Strichartzlin}. Therefore,  estimate \eqref{Strichartzlin3} implies that
\begin{equation} \label{BilinR2.iv.2}
I_{HH \to H}^4 \lesssim \|u\|_{L^2}\|v\|_{L^2}\|w\|_{L^2}.
\end{equation}

Finally, we conclude the proof of estimate \eqref{BilinR2.1} gathering estimates \eqref{BilinR2.4}, \eqref{BilinR2.400}, \eqref{BilinR2.7}, \eqref{BilinR2.8}, \eqref{BilinR2.12}, \eqref{BilinR2.13}, \eqref{BilinR2.i.8}, \eqref{BilinR2.ii.1} and \eqref{BilinR2.iv.2}.
\end{proof}

At this point, we observe that the proof of Theorem \ref{theoR2} follows from Proposition \ref{BilinR2} and the linear estimates \eqref{prop1.1.2}, \eqref{prop1.2.1} and \eqref{prop1.3b.1} by using a fixed point argument in a closed ball of $X^{s,\frac12+\delta}_T$ (see for example \cite{MST} for more details).

\section{Bilinear estimate in $\mathbb R \times \mathbb T$}

The main result of this section is stated below. 
\begin{proposition} \label{BilinRT}
Let $s \ge 1$. Then, there exists $\delta>0$ such that 
\begin{equation} \label{BilinRT.1}
\|\partial_x(uv)\|_{X^{s,-\frac12+2\delta}} \lesssim \|u\|_{X^{s,\frac12+\delta}}\|v\|_{X^{s,\frac12+\delta}},
\end{equation}
for all $u, \ v: \mathbb R \times \mathbb T \times \mathbb R \rightarrow \mathbb R$ such that $u, \ v \in X^{s,\frac12+\delta}$.
\end{proposition}

\begin{proof} By duality, it suffices to prove that 
\begin{equation} \label{BilinRT.2}
J\lesssim \|u\|_{L^2_{x,y,t}}\|v\|_{L^2_{x,y,t}}\|w\|_{L^2_{x,y,t}},
\end{equation}
where 
\begin{displaymath} 
J=\sum_{q,  q_1 \in \mathbb Z^2}\int_{\mathbb R^4}\Gamma^{\xi_1,q_1,\tau_1}_{\xi,q,\tau}\widehat{w}(\xi,q,\tau)\widehat{u}(\xi_1,q_1,\tau_1)\widehat{v}(\xi_2,q_2,\tau_2) d\nu,
\end{displaymath}
$\widehat{u}$, $\widehat{v}$ and $\widehat{w}$ are nonnegative functions, and we used the following notations
\begin{equation} \label{BilinRT.3}
\begin{split}
&\Gamma^{\xi_1,q_1,\tau_1}_{\xi,q,\tau}=|\xi| \langle |(\xi,q)| \rangle^s\langle \sigma\rangle^{-\frac12+2\delta}\langle |(\xi_1,q_1)| \rangle^{-s}\langle \sigma_1 \rangle^{-\frac12-\delta}
\langle |(\xi_2,q_2)| \rangle^{-s}\langle \sigma_2 \rangle^{-\frac12-\delta}, \\
&d\nu=d\xi d\xi_1d\tau d\tau_1, \quad \xi_2=\xi-\xi_1, \ q_2=q-q_1, \ \tau_2=\tau-\tau_1, \\ 
&\sigma=\tau-w(\xi,q) \quad \text{and} \quad \sigma_i=\tau_i-w(\xi_i,q_i), \ i=1,2.
\end{split}
\end{equation}

By using dyadic decompositions on the spatial frequencies of $u$, $v$ and $w$, we rewrite $J$ as 
\begin{equation} \label{BilinRT.4}
J=\sum_{N_1,N_2,N}J_{N,N_1,N_2} ,
\end{equation}
where 
\begin{displaymath} 
J_{N,N_1,N_2}=\sum_{q,  q_1 \in \mathbb Z^2}\int_{\mathbb R^4}\Gamma^{\xi_1,q_1,\tau_1}_{\xi,q,\tau}\widehat{P_Nw}(\xi,q,\tau)\widehat{P_{N_1}u}(\xi_1,q_1,\tau_1)\widehat{P_{N_2}v}(\xi_2,q_2,\tau_2) d\nu .
\end{displaymath}
Now, we use the decomposition
\begin{equation} \label{BilinRT.5}
J=J_{LL \to L}+J_{LH \to H}+J_{HL \to H}+J_{HH \to L}+J_{HH \to H},
\end{equation}
where $J_{LL \to L}$, $J_{LH \to H}$, $J_{HL \to H}$, $J_{HH \to L}$, respectively $J_{HH \to H}$, denote the $Low\times Low \to Low$, $Low \times High \to High$, $High \times Low \to High$, $High \times High \to Low$, respectively $High \times High \to High$ contributions for $J$ as defined in the proof of Proposition \ref{BilinR2}. \\

\noindent \textit{1. Estimate for $J_{LH \to H}+J_{HL \to H}+J_{HH \to L}$.} Since Proposition \ref{BilinStrichartzI} also holds in the $\mathbb R \times \mathbb T$ case, we deduce arguing as in \eqref{BilinR2.7}, \eqref{BilinR2.8} and \eqref{BilinR2.12} that 
\begin{equation} \label{BilinRT.6}
J_{LH \to H}+J_{HL \to H}+J_{HH \to L} \lesssim \|u\|_{L^2}\|v\|_{L^2}\|w\|_{L^2}.
\end{equation}

\noindent \textit{2. Estimate for $J_{HH \to H}$.} We recall that $N\sim N_1 \sim N_2$ in this case. We divide the integration domain in several regions. \\

\noindent \textit{2.1 Estimate for  $J_{HH \to H}$ in the region $|\xi| \le 100$.} We denote by $J_{HH \to H}^1$ the restriction of $J_{HH \to H}$ to the region $|\xi|\le 100$ and use dyadic decompositions on the variables $\sigma$, $\sigma_1$, $\sigma_2$ and $\xi$, so that 
\begin{equation} \label{BilinRT.7} 
J_{N,N_1,N_2}=\sum_{k \ge 0}\sum_{L,L_1,L_2}J_{N,N_1,N_2,k}^{L,L_1,L_2},
\end{equation}
where $J_{N,N_1,N_2,k}^{L,L_1,L_2}$ is given by the expression
\begin{displaymath}
\sum_{q,  q_1 \in \mathbb Z^2}\int_{\mathcal{E}_k}\Gamma^{\xi_1,q_1,\tau_1}_{\xi,q,\tau}\widehat{P_NQ_Lw}(\xi,q,\tau)\widehat{P_{N_1}Q_{L_1}u}(\xi_1,q_1,\tau_1)\widehat{P_{N_2}Q_{L_2}v}(\xi_2,q_2,\tau_2) d\nu,
\end{displaymath}
with $\mathcal{E}_k=\{(\xi,\xi_1,\tau,\tau_1) \in \mathbb R^4 \ : \ 2^{-(k+1)}100 \le |\xi| \le 2^{-k}100 \}$. Thus, by using the Cauchy-Schwarz inequality, we get that
\begin{equation} \label{BilinRT.8}
J_{N,N_1,N_2,k}^{L,L_1,L_2} \lesssim 2^{-k}N_1^{-s}L^{-\frac12+2\delta}L_1^{-\frac12-\delta}L_2^{-\frac12-\delta}\|(P_{N_1}Q_{L_1}u) (P_{N_2}Q_{L_2}v)\|_{L^2}\|w\|_{L^2}.
\end{equation}
Next, we argue as in \eqref{BilinStrichartzI3} to estimate $\|(P_{N_1}Q_{L_1}u) (P_{N_2}Q_{L_2}v)\|_{L^2}$. Moreover, we observe that
\begin{displaymath} 
\Big|\frac{\partial^2H}{\partial \xi_1^2} (\xi,\xi_1,q,q_1)\Big|=6|\xi| \sim 2^{-k}.
\end{displaymath}
Thus, it follows from Lemma \ref{basicI}, estimates \eqref{basicIII1}, \eqref{BilinStrichartzI3} and \eqref{BilinStrichartzI4} that 
\begin{equation} \label{BilinRT.9} 
\begin{split}
\|(P_{N_1}Q_{L_1}u) &(P_{N_2}Q_{L_2}v)\|_{L^2} \\ & \lesssim 2^{\frac{k}4}N_1^{\frac12}(L_1 \wedge L_2)^{\frac12}(L_1 \vee L_2)^{\frac14} \|P_{N_1}Q_{L_1}u\|_{L^2}\|P_{N_2}Q_{L_2}v\|_{L^2}.
\end{split}
\end{equation}
Therefore, we deduce combining \eqref{BilinRT.8} and \eqref{BilinRT.9} and summing over $L, \ L_1, \ L_2$, $N\sim N_1\sim N_2$ and $k \in \mathbb N$ that
\begin{equation} \label{BilinRT.10} 
J^1_{HH \to H} \lesssim \|u\|_{L^2}\|v\|_{L^2}\|w\|_{L^2}.
\end{equation}

\noindent \textit{2.2 Estimate for $J_{HH \to H}$ in the region $|\xi| \ge 100$, and $|\xi_1| \wedge |\xi_2| \le 100$.} We denote by $J_{HH \to H}^2$ the restriction of $J_{HH \to H}$ to this region and use dyadic decompositions on the variables $\sigma$, $\sigma_1$, $\sigma_2$, so that 
\begin{equation} \label{BilinRT.11} 
J_{N,N_1,N_2}=\sum_{L,L_1,L_2}J_{N,N_1,N_2}^{L,L_1,L_2},
\end{equation}
where $J_{N,N_1,N_2}^{L,L_1,L_2}$ is given by the expression
\begin{equation} \label{BilinRT.110}
\sum_{q,  q_1 \in \mathbb Z^2}\int_{\mathbb R^4}\Gamma^{\xi_1,q_1,\tau_1}_{\xi,q,\tau}\widehat{P_NQ_Lw}(\xi,q,\tau)\widehat{P_{N_1}Q_{L_1}u}(\xi_1,q_1,\tau_1)\widehat{P_{N_2}Q_{L_2}v}(\xi_2,q_2,\tau_2) d\nu.
\end{equation}
Thus, the Caucy-Schwarz inequality implies that 
\begin{equation} \label{BilinRT.12}
J_{N,N_1,N_2}^{L,L_1,L_2} \lesssim L^{-\frac12+2\delta}L_1^{-\frac12-\delta}L_2^{-\frac12-\delta}\|(P_{N_1}Q_{L_1}u) (P_{N_2}Q_{L_2}v)\|_{L^2}\|w\|_{L^2},
\end{equation}
where we used the bound $|\xi| \le N \sim N_1\sim N_2$ and $s\ge 1$. This time, we observe that 
\begin{displaymath} 
\Big|\frac{\partial^2H}{\partial q_1^2} (\xi,\xi_1,q,q_1)\Big|=2|\xi| \gtrsim 1.
\end{displaymath}
in order to estimate $\|(P_{N_1}Q_{L_1}u) (P_{N_2}Q_{L_2}v)\|_{L^2}$. Then, since $|\xi_1|\wedge |\xi_2| \le 1$, it follows from Lemma \ref{basicI}, estimates \eqref{basicIII2}, \eqref{BilinStrichartzI3} and \eqref{BilinStrichartzI4}   that
\begin{equation} \label{BilinRT.13} 
\begin{split}
\|(P_{N_1}Q_{L_1}u) &(P_{N_2}Q_{L_2}v)\|_{L^2} \\ & \lesssim (L_1 \wedge L_2)^{\frac12}\big(1+(L_1 \vee L_2)^{\frac14}\big) \|P_{N_1}Q_{L_1}u\|_{L^2}\|P_{N_2}Q_{L_2}v\|_{L^2}.
\end{split}
\end{equation}
Therefore, we deduce combining \eqref{BilinRT.12} and \eqref{BilinRT.13} and summing over $L, \ L_1, \ L_2$ and $N\sim N_1\sim N_2$ (here we use the Cauchy-Schwarz inequality in $N_1$) that
\begin{equation} \label{BilinRT.14} 
J^2_{HH \to H} \lesssim \|u\|_{L^2}\|v\|_{L^2}\|w\|_{L^2}.
\end{equation}

\noindent \textit{2.3 Estimate for $J_{HH \to H}$ in the region $|\xi_i| \ge 100$ for $i=1,2,3$.}  We denote
 by $J_{HH \to H}^3$ the restriction of $J_{HH \to H}$ to this region. Once again, we use dyadic decompositions  on the variables $\sigma$, $\sigma_1$ and $\sigma_2$ as in \eqref{BilinRT.11}. In order to simplify the notations, we will denote $(\xi,q,\tau)=(\xi_0,q_0,\tau_0)$. Next, for $0<\delta \ll 1$, we split the integration domain in the following subregions 
 \begin{displaymath}
 \begin{split}
 \mathcal{F}_{3.1} =\big\{& (\xi,\xi_1,\tau,\tau_1,q,q_1) \in \mathbb R^4 \times \mathbb Z^2 \ : |\xi_i| \ge 100, \ \forall i \in \{0,1,2\} \\ &\ \text{and} \ \exists \, (i,j) \in \{0,1,2\}\ \text{with} \ \big| |(\xi_i,q_i)|^2 -|(\xi_j,q_j)|^2\big| \ge NL^{6\delta} \ \big\},
 \end{split}
 \end{displaymath}
 \begin{displaymath}
 \begin{split}
 \mathcal{F}_{3.2} =\big\{& (\xi,\xi_1,\tau,\tau_1,q,q_1) \in \mathbb R^4 \times \mathbb Z^2 \ : |\xi_i| \ge 100, \ \forall i \in \{0,1,2\} \\ &\ \text{and} \  \big| |(\xi_i,q_i)|^2 -|(\xi_j,q_j)|^2\big| \le NL^{6\delta}, \ \forall \, (i,j) \in \{0,1,2\} \ \big\}.
 \end{split}
 \end{displaymath}
 and denote by $J_{HH \to H}^{3,1}$, respectively $J_{HH \to H}^{3,2}$, the restriction of $J_{HH \to H}$ to $\mathcal{F}_{3.1}$, respectively $\mathcal{F}_{3.2}$. \\ 
 
\noindent \textit{2.3.1 Estimate for $J_{HH \to H}^{3,1}$.} Without loss of generality, we can assume that 
\begin{equation} \label{BilinRT.15} 
\big| |(\xi,q)|^2 -|(\xi_1,q_1)|^2\big| \ge NL^{6\delta} .
\end{equation}
By using the Cauchy-Schwarz inequality and the fact that $|\xi| \lesssim N \sim N_1 \sim N_2$ and $s \ge 1$, we obtain that 
\begin{equation} \label{BilinRT.16}
J_{N,N_1,N_2}^{L,L_1,L_2} \lesssim L^{-\frac12+2\delta}L_1^{-\frac12-\delta}L_2^{-\frac12-\delta}
\big\|(\widetilde{P_{N_1}Q_{L_1}u}) (P_{N}Q_{L}w)\big\|_{L^2} \|P_{N_2}Q_{L_2}v\|_{L^2},
\end{equation}
where $\tilde{f}(\xi,q,\tau)=f(-\xi,-q,-\tau)$.  Moreover, we observe arguing exactly as in the proof of Proposition \ref{BilinStrichartzI} and by using \eqref{BilinRT.15} that
\begin{equation} \label{BilinRT.17}
\begin{split}
\|(&\widetilde{P_{N_1}Q_{L_1}u}) (P_{N}Q_{L}w)\|_{L^2} \\ & \lesssim
\frac{(N_1 \wedge N_2)^{\frac12}}{N^{\frac12}L^{3\delta}}(L_1 \vee
L)^{\frac12}(L_1 \wedge L)^{\frac12}\|P_{N_1}Q_{L_1}u\|_{L^2}
\|P_{N}Q_{L}w\|_{L^2}.
\end{split}
\end{equation}
Therefore, we deduce combining \eqref{BilinRT.16} and \eqref{BilinRT.17} and summing over $L, \ L_1, \ L_2$ and $N \sim N_1 \sim N_2$ (by using the Cauchy-Schwarz inequality in $N$) that 
\begin{equation} \label{BilinRT.18} 
J^{3,1}_{HH \to H} \lesssim \|u\|_{L^2}\|v\|_{L^2}\|w\|_{L^2}.
\end{equation}

\noindent \textit{2.3.2 Estimate for $J_{HH \to H}^{3,2}$.} In the region $\mathcal{F}_{3,2}$, it holds that 
\begin{equation} \label{BilinRT.19}
\big| |(\xi_i,q_i)|^2 -|(\xi_j,q_j)|^2\big| \le NL^{6\delta}, \ \forall \, (i,j) \in \{0,1,2\} .
\end{equation} 
Then, we deduce from the definition of $H$ in \eqref{Resonance}, the definition $|(\xi,q)|=\sqrt{3\xi^2+q^2}$ and the assumptions \eqref{BilinRT.19} that 
\begin{equation} \label{BilinRT.20}
\begin{split}
H(\xi,\xi_1,q,q_1)&=(\xi-\xi_1-\xi_2)|(\xi_{i_0},q_{i_0})|^2-6\xi\xi_1\xi_2+\Theta(\xi,\xi_1,q,q_1) \\ 
&=-6\xi\xi_1\xi_2+\Theta(\xi,\xi_1,q,q_1),
\end{split}
\end{equation}
for $i_0 \in \{1,2,3 \}$ such that $|\xi_{i_0}|=\max\{|\xi_j|  \ : \ j=1,2,3\}$ and $\Theta(\xi,\xi_1,q,q_1)$ satisfies 
\begin{equation} \label{BilinRT.21} 
\big| \Theta(\xi,\xi_1,q,q_1)\big| \le \sum_{i \neq i_0}|\xi_i|\big| |(\xi_i,q_i)|^2 -|(\xi_j,q_j)|^2\big| \le |\xi_{med}|NL^{6\delta}.
\end{equation}
It follows combining \eqref{BilinRT.20} and \eqref{BilinRT.21} that 
\begin{equation} \label{BilinRT.22} 
\Big| H(\xi,\xi_1,q,q_1) \Big| \ge |\xi_{med}| \Big(6|\xi_{max}||\xi_{min}|-NL^{6\delta} \Big).
\end{equation}
Then, we subdivide the region $\mathcal{F}_{1.2}$ in the following subregions 
\begin{displaymath}
 \mathcal{F}_{3.2.1} =\big\{ (\xi,\xi_1,\tau,\tau_1,q,q_1) \in \mathcal{F}_{1.2} \ : |\xi_{max}||\xi_{min}| \ge NL^{6\delta} \big\},
 \end{displaymath}
 \begin{displaymath}
 \mathcal{F}_{3.2.2} =\big\{ (\xi,\xi_1,\tau,\tau_1,q,q_1) \in \mathcal{F}_{1.2} \ : |\xi_{max}||\xi_{min}| \le NL^{6\delta} \big\},
 \end{displaymath}
 and denote by $J_{HH\to H}^{3,2,1}$, respectively $J_{HH\to H}^{3,2,2}$, the restriction of $J_{HH \to H}^{3,2}$ to $\mathcal{F}_{3.2.1}$, respectively $ \mathcal{F}_{3.2.2}$. \\
 
 \noindent \textit{2.3.2.1 Estimate for $J_{HH\to H}^{3,2,1}$.} Due to \eqref{BilinRT.22}, we have that 
 \begin{equation} \label{BilinRT.23} 
 \max\{|\sigma|,|\sigma_1|,|\sigma_2| \} \gtrsim |\xi_{min}||\xi_{max}|^2,
 \end{equation}
 in $\mathcal{F}_{3.2.1}$. Without loss of generality\footnote{In the other cases we need to interpolate \eqref{BilinRT.25} with \eqref{BilinStrichartzI0} as previously.}, we assume that $\max\{|\sigma|,|\sigma_1|,|\sigma_2| \}=|\sigma|$.
Then, by using the Cauchy-Schwarz inequality, we deduce that 
\begin{equation} \label{BilinRT.24} 
J_{N,N_1,N_2}^{L,L_1,L_2} \lesssim  N_1^{-\frac12} L^{-\delta}L_1^{-\frac12-\delta}L_2^{-\frac12-\delta}\|(P_{N_1}Q_{L_1}u) (P_{N_2}Q_{L_2}v)\|_{L^2}\|w\|_{L^2}.
\end{equation}
where we used that $\big| |\xi| N_1^{\frac12-s} \big(|\xi_{max}|^2|\xi_{min}| \big)^{-\frac12+3\delta}\big| \lesssim 1$ for $s \ge 1$ and $0<\delta \ll 1$. Moreover, we use that 
\begin{displaymath} 
\Big|\frac{\partial^2H}{\partial \xi_1^2} (\xi,\xi_1,q,q_1)\Big|=6|\xi| \gtrsim 1,
\end{displaymath}
Lemma \ref{basicI}, estimates \eqref{basicIII1}, \eqref{BilinStrichartzI3} and \eqref{BilinStrichartzI4} lead to 
\begin{equation} \label{BilinRT.25} 
\begin{split}
\|(P_{N_1}Q_{L_1}u) &(P_{N_2}Q_{L_2}v)\|_{L^2} \\ & \lesssim N_1^{\frac12}(L_1 \wedge L_2)^{\frac12}(L_1 \vee L_2)^{\frac14} \|P_{N_1}Q_{L_1}u\|_{L^2}\|P_{N_2}Q_{L_2}v\|_{L^2}.
\end{split}
\end{equation}
We deduce combining \eqref{BilinRT.24} and \eqref{BilinRT.25} and summing over $L,$ $L_1$, $L_2$ and using the Cauchy-Schwarz inequality in $ N_1 \sim N_2$ that 
\begin{equation} \label{BilinRT.26} 
J_{HH\to H}^{3,2,1} \lesssim \|u\|_{L^2}\|v\|_{L^2}\|w\|_{L^2}.
\end{equation}

\noindent \textit{2.3.2.2 Estimate for $J_{HH\to H}^{3,2,2}$.} This time, we perform also dyadic decompositions in the $\xi_1$, $\xi_2$ and $\xi$ variables. We denote by $R_{K}$ the Littlewood-Paley projectors , \textit{i.e.} $R_K$ is defined by 
$R_Ku=\mathcal{F}_x^{-1}\big(\phi(K^{-1}\xi)\mathcal{F}_x(u) \big)$, for any dyadic number $K \ge 1$. Then, we have that
\begin{equation} \label{BilinRT.27}
J_{N,N_1,N_2}^{L,L_1,L_2}=\sum_{100 \le K_1,K_2,K_3 \lesssim N}J_{N,N_1,N_2}^{L,L_1,L_2}(K_1,K_2,K_3),
\end{equation}
where $J_{N,N_1,N_2}^{L,L_1,L_2}(K_1,K_2,K_3)$ is defined by the expression
\begin{displaymath} 
\begin{split}
J_{N,N_1,N_2}^{L,L_1,L_2}(K_1,K_2,K_3)&=\sum_{q,  q_1 \in \mathbb Z^2}\int_{\mathbb R^4}\Gamma^{\xi_1,q_1,\tau_1}_{\xi,q,\tau}\big(P_NQ_LR_Kw\big)^{\wedge}(\xi,q,\tau)\\ & \times\big(P_{N_1}Q_{L_1}R_{K_1}u\big)^{\wedge}(\xi_1,q_1,\tau_1)\big(P_{N_2}Q_{L_2}R_{K_2}v\big)^{\wedge}(\xi_2,q_2,\tau_2) d\nu.
\end{split}
\end{displaymath}
By using the Cauchy-Schwarz inequality, we can bound $J_{N,N_1,N_2}^{L,L_1,L_2}(K_1,K_2,K_3)$  by
\begin{equation} \label{BilinRT.28}
 KK_{min}^{-1}K_{max}^{-1}N^{1-s} L^{-\frac12+8\delta}L_1^{-\frac12-\delta}L_2^{-\frac12-\delta}\|(P_{N_1}Q_{L_1}u) (P_{N_2}Q_{L_2}v)\|_{L^2}\|w\|_{L^2}.
\end{equation}
since $K_{min}K_{max} \lesssim NL^{6\delta}$ in the region $\mathcal{F}_{3,2,2}$. Moreover, noticing that 
\begin{displaymath} 
\Big|\frac{\partial^2H}{\partial q_1^2} (\xi,\xi_1,q,q_1)\Big|=6|\xi| \gtrsim K,
\end{displaymath}
Lemma \ref{basicI}, estimates \eqref{basicIII2}, \eqref{BilinStrichartzI3} and \eqref{BilinStrichartzI4} yield
\begin{equation} \label{BilinRT.29} 
\begin{split}
&\|(P_{N_1}Q_{L_1}u) (P_{N_2}Q_{L_2}v)\|_{L^2} \\ & \lesssim (K_1 \wedge K_2)^{\frac12}(L_1 \wedge L_2)^{\frac12}(1+K^{-\frac14}(L_1 \vee L_2)^{\frac14}) \|P_{N_1}Q_{L_1}u\|_{L^2}\|P_{N_2}Q_{L_2}v\|_{L^2}.
\end{split}
\end{equation}
Now, we observe that 
\begin{equation} \label{BilinRT.30} 
K (K_1 \wedge K_2)^{\frac12}K_{min}^{-1}K_{max}^{-1}\lesssim K_{min}^{-\frac12}.
\end{equation}
Assume without loss of generality that $K_{min}=K$. Therefore, it follows combining \eqref{BilinRT.27}--\eqref{BilinRT.30}, summing over $L$, $L_1$, $L_2$ and $K_{min}$ and applying Cauchy-Schwarz in $K_1 \sim K_2$ and in $N_1 \sim N_2$ that 
\begin{equation} \label{BilinRT.31} 
\begin{split}
J_{HH\to H}^{3,2,2} &\lesssim \sum_{N \sim N_1 \sim N_2}\sum_{100 \le K_1 \sim K_2 \lesssim N}\|P_{N_1}R_{K_1}u\|_{L^2}\|P_{N_2}R_{K_2}v\|_{L^2}\|P_{N}w\|_{L^2}
\\ & \lesssim \sum_{ N_1 \sim N_2}\Big( \sum_{K_1 \le N_1}\|P_{N_1}R_{K_1}u\|_{L^2}^2\Big)^{\frac12}
\Big( \sum_{K_2 \le N_2}\|P_{N_2}R_{K_2}v\|_{L^2}^2\Big)^{\frac12}\|w\|_{L^2} \\ & \lesssim 
\Big( \sum_{N_1}\|P_{N_1}u\|_{L^2}^2\Big)^{\frac12}\Big( \sum_{N_2}\|P_{N_2}v\|_{L^2}^2\Big)^{\frac12}\|w\|_{L^2} \\ & \lesssim  \|u\|_{L^2}\|v\|_{L^2}\|w\|_{L^2}.
\end{split}
\end{equation}

Thus, we deduce combining \eqref{BilinRT.10}, \eqref{BilinRT.14}, \eqref{BilinRT.18}, \eqref{BilinRT.26} and \eqref{BilinRT.31} that 
\begin{equation} \label{BilinRT.32} 
J_{HH\to H} \lesssim \|u\|_{L^2}\|v\|_{L^2}\|w\|_{L^2}.
\end{equation}

\noindent \textit{2.3 Estimate for $J_{LL \to L}$.} We get arguing exactly as in the cases \textit{2.1}  and   \textit{2.2} that 
\begin{equation} \label{BilinRT.33} 
J_{LL\to L} \lesssim \|u\|_{L^2}\|v\|_{L^2}\|w\|_{L^2}.
\end{equation}

Finally, we conclude the proof of estimate \eqref{BilinRT.1} gathering  \eqref{BilinRT.5}, \eqref{BilinRT.6}, \eqref{BilinRT.32} and \eqref{BilinRT.33}.
\end{proof} 
 
 We observe that the proof of Theorem \ref{theoRT} follows from Proposition \ref{BilinRT} and the linear estimates \eqref{prop1.1.2}, \eqref{prop1.2.1} and \eqref{prop1.3b.1} by using a fixed point argument in a closed ball of $X^{s,\frac12+\delta}_T$ (see for example \cite{MST} for more details).
  \section{Global existence in  $ H^s(\R^3) $ for $ s>1$}
 In this section we prove the global well-posedness in $ H^s(\R^3) $ for $ s>1$. To this aim we combine  the conservation laws \eqref{M} and \eqref{H}, a well-posedness result in the Besov space $ B^{1,1}_2 (\R^3)$  and follow ideas in \cite{Bou} (see \cite{Plan} for the same kind of arguments). One crucial tool will also be  the atomic spaces $ U^2 $ and $ V^2 $  introduced by Koch-Tataru in \cite{KT}. 
 Recall that  the Besov space $ B^{1,1}_2(\R^3) $ is the  space of all functions $ g\in {\mathcal S}'(\R^3) $ such that 
 \begin{equation}\label{defB}
 \|g\|_{B^{1,1}_2} := \sum_{N} N \|P_N g \|_{L^2} < \infty \; , 
 \end{equation}
where the Fourier projector $ P_N $ is the $ \R^3 $-version of the one defined in \eqref{proj}.
 
 Before stating the local existence theorem let us give the definition of the doubling time that will appear in the statement of this theorem. Let be given  a Cauchy problem locally well-posed in some Banach space $ B $ with a minimum time of existence depending on the $ B$-norm of the initial data. For any $ u_0\in B$ we call \lq\lq doubling time\rq\rq , the infinite or finite positive real number
 $$
 T(u_0) = \sup \Bigl\{t>0 \, :\, \|u(\theta)\|_{B} \le 2 \|u_0\|_{B}  \mbox{ on } [0,t] \; \Bigr\} \; .
 $$
 \begin{theorem}\label{prop2}
 The Cauchy problem  associated to \eqref{ZK} is locally well-posed in $ H^s(\R^3) $ for $ s>1 $. Moreover,  there exists $ C>0 $ such that for any $ u_0\in H^s(\R^3) $, the doubling time satisfies 
 \begin{equation} \label{Z3}
 T(u_0)\ge  \frac{C}{\|u_0\|_{B^{1,1}_2}^2} \; .
 \end{equation}
 \end{theorem}
 \begin{remark} The local well-posedness of ZK in $ H^s(\R^3) $ for $ s>1$ was already proven in  \cite{RV}. The only new result here is the estimate from below of the doubling time.
 \end{remark}
 With Theorem \ref{prop2} in hand we will now prove the Theorem \ref{theo3}. The proof of  Theorem \ref{prop2} is postponed at the end of this section.  \\
 {\bf Proof of Theorem \ref{theo3} :} Let us fix $ s>1 $.  For any $ g\in H^s(\R^3) $ and any $ k\ge 1 $  it holds
 \begin{eqnarray*}
  \|g\|_{B^{1,1}_2}&= &\sum_{j=0}^{k-1} 2^{j} \| P_{2^j} g \|_{L^2}+ \sum_{j=k}^\infty  2^{j(1-s)}  2^{js} \| P_{2^j} g \|_{L^2}\\
  & \lesssim & \sqrt{k} \|g\|_{H^1} + 2^{k(1-s)} \|g\|_{H^s} \; .
 \end {eqnarray*}
 Therefore, taking $ k=\frac{\ln(1+\|g\|_{H^s})}{(s-1)\ln 2} $ we get 
  \begin{equation}\label{z1}
 \|g\|_{B^{1,1}_2}\le C_s \, \Bigl( 1+ \|g\|_{H^1}  \ln (1+\|g\|_{H^s})^{1/2} \Bigr)
 \end{equation} 
 for some $ C_s>0 $. \\
 Now, let $ u_0\in H^s(\R^3) $ and $ u $ be  the solution of ZK emanating from $ u_0 $. Combining Proposition \ref{prop2} and \eqref{z1}  we get 
$$
 T(u_0) \ge  \frac{C}{1+\|u_0\|_{H^1}^2 \ln (1+\|u_0\|_{H^s})}\; .
 $$
 If $ T(u_0)=+\infty $ then we are done. Otherwise we set $u_1:= u(T(u_0)) $. In the same way as above we have 
 $$
  T(u_1) \ge  \frac{C}{1+\|u_1\|_{H^1}^2  \ln (1+\|u_1\|_{H^s})}\; .
 $$
 Now,  from the definition of the doubling time,  it holds 
  $\|u_1\|_{H^s}= 2 \|u_0\|_{H^s}$ and from the conservation of the quantities $ M(u) $ and $ H(u) $ and classical Sobolev inequalities we infer that 
 $$
 \|u_1\|_{H^1}^2 \le C'  E(u_1) =C' E(u_0) \; ,
 $$
 for some positive constant $ C' $ independent of $ u_1$. 
 Therefore, setting $ E_0:=E(u_0) $, we obtain 
 $$
  T(u_1) \ge \frac{C}{1+C' E_0 \ln (1+2\|u_0\|_{H^s})}\; .
 $$
 Repeating this argument n-times (assuming that all doubling time $T(u_k) $, $ k=1,2,..,n-1$, are finite, since otherwise we are done), we get 
  \begin{equation}\label{z2}
  T(u_n) \ge \frac{C}{1+C' E_0 \ln (1+2^n\|u_0\|_{H^s})}\gtrsim\frac{1}{n}\; .
  \end{equation}
  Since $ \sum 1/n =+\infty $ this ensures that for any given $ T>0 $ there exists $ n\ge 1 $ such that 
  $\displaystyle \sum_{ k=0}^{n-1} T(u_k)>T $ 	and thus the solution is global in time. 
  \begin{remark}
  Actually, it is not too hard to check that the lower bound \eqref{z2}  leads to a double exponential upper bound on the solution $ u$, i.e. there exists constants $ K_1$, $ K_2 $ and $ K_3 $ only depending on $ \|u_0\|_{H^s} $ such that for all $ t\ge 0 $,
  $$
  \|u(t)\|_{H^s} \le K_1 \exp\Bigl(K_2\exp(K_3 t) \Bigr) \; .
  $$
  \end{remark}
\subsection{Proof of Theorem \ref{prop2}}
\subsubsection{Resolution spaces}
We start by recalling the definition of the  function spaces $ U^2 $ and $ V^2 $ (see \cite{KT} and  \cite{HHK}).
\begin{definition}
Let ${\mathcal Z} $ be the set of finite partitions $ -\infty=t_0<t_1<\cdot\cdot <t_K=+\infty $. 
For $ \{t_k\}_{k=0}^K \in {\mathcal Z}$ and $\{\phi_k\}_{k=0}^{K-1} \subset L^2(\R^3)$ with  $\sum_{k=0}^{K-1} \| \phi_k\|_{L^2}^2 =1 $ and $ \phi_0=0 $ we call the function $ a\, :\, \R \to L^2(\R^3) $ given by 
$$
a=\sum_{k=1}^K {1\! \! 1}_{[t_{k-1},t_k)} \phi_{k-1}
$$
a $ U^2 $-atom and we define the atomic space 
$$
U^2:=\Bigl\{u=\sum_{j=1}^\infty \lambda_j a_j \ : \ a_j \; U^2\mbox{-atom} \mbox{ and } \lambda_j\in \R \mbox{ with } 
 \sum_{j=1}^\infty |\lambda_j | <\infty  \Bigr\} 
 $$
 with norm 
 $$
 \|u\|_{U^2} := \inf \Bigl\{ \sum_{j=1}^\infty |\lambda_j |\ : \  u=\sum_{j=1}^\infty \lambda_j a_j \mbox{ with }\, \lambda_j \in \R \mbox{ and } a_j \, U^2\mbox{-atom} \Bigr\} 
 $$
 The function space $ V^2$ is defined as the normed space of all functions $ v\, :\, \R \to L^2(\R^3) $ such that $\lim_{t\to \mp \infty}  v(t) $ exists and for which the norm 
 $$
 \|v\|_{V^2}:=\sup_{\{t_k\}_{k=0}^K \in {\mathcal Z}} \Bigl( \sum_{k=1}^K \|v(t_k)-v(t_{k-1})\|_{L^2}^2 \Bigr)^{1/2}
 $$
 is finite, where we use the convention that $v(-\infty)=\lim_{t\to -\infty} v(t) $ and $ v(+\infty)=0 $. 
\end{definition}
The  spaces $ U^2 $ and $ V^2 $ are Banach spaces. They   will serve as   substitutes of the Besov  type spaces $ \tilde{B}^{1/2,1}_2(L^2(\R^3))  $ and $  \tilde{B}^{1/2,\infty}_2 (L^2(\R^3))$ that where first used in \cite{tataru} in the context of Bourgain's method. Denoting by $ \Delta_j $ the Fourier multiplier by\footnote{See Section 2 for the definition of $ \phi $ and $\eta$.}  $ \phi(2^j \tau) $ for $ j\ge 1 $ and $ \eta(\tau)$ for $ j=0$, these last spaces are respectively endowed with the  norms
$$
\|u\|_{\tilde{B}^{1/2,1}_2 (L^2(\R^3))} :=\sum_{j\ge 0} 2^{j/2} \| \Delta_j u \|_{L^2(\R^4)}$$  
and 
$$
 \|u\|_{\tilde{B}^{1/2,\infty}_2 (L^2(\R^3))} :=\sup_{j\ge 0} 2^{j/2} \| \Delta_j u \|_{L^2(\R^4)} \; .
$$
The crucial point for us will be that, since $ V^2\hookrightarrow  L^\infty_t L^2 $, for a smooth function $ \psi\in C^\infty_c (\R) $ and any $ 0<T<1 $,  it holds
\begin{equation} \label{zzz}
\| \psi(\cdot/T) f \|_{L^2(\R^4))} \lesssim T^{1/2} \| f\|_{V^2}, \quad \forall f\in  C^\infty_c (\R^4),
\end{equation}
whereas we only have 
$$
\| \psi(\cdot/T) f \|_{L^2(\R^4))} \lesssim T^{1/2} |\ln T| \| f\|_{ \tilde{B}^{1/2,\infty}_2 (L^2(\R^3))} , \quad \forall f\in  C^\infty_c (\R^4) \; .
$$
This last inequality would  lead to a lower bound 
$$
T(u_0) \gtrsim \frac{1}{\|u_0\|_{B^{1,1}_2}^2 | \ln (\|u_0\|_{B^{1,1}_2})|^2}\;
$$
of the doubling time that will not be sufficient to get the global existence result. This is the reason why we will work with the couple of spaces $ U^2 $ and   $ V^2 $ 
and not  with the more usual couple of spaces $  \tilde{B}^{1/2,1}_2 (L^2(\R^3)) $ and $   \tilde{B}^{1/2,\infty}_2 (L^2(\R^3)) $.\\
Then denoting by $S(t):=e^{-t\partial_x\Delta} $ the linear group associated with ZK, we define the spaces 
 $$ U^2_S= S(\cdot) U^2 \mbox{  with norm } \|u\|_{U^2_S}=\| S(-\cdot) u\|_{U^2} $$
  $$\hspace*{-6mm} \mbox{and  }V^2_S= S(\cdot) V^2  \mbox{  with norm }  \|u\|_{V^2_S}
 =\| S(-\cdot) u\|_{V^2} \;  . $$
 The properties of these spaces  we need in the sequel are summarized in the following propositions (see \cite{HHK}).
 \begin{proposition} \label{prop3}
 Let $ \psi \in C^\infty_c(\R)$ then  
 $$
 \| \psi S(\cdot) u_0\|_{U^2_S}\lesssim \|u_0 \|_{L^2} , \quad  \forall u_0\in L^2(\R^3) 
$$
and 
$$
\Bigl\|\psi(t)\int_0^t S(t-t') f(t',\cdot) \, dt' \Bigr\|_{U^2_S} \lesssim \sup_{\|v\|_{V^2_S}=1} \Bigl| \int_{\R^4} f v \Bigr| , \quad  \forall f \in C^\infty_c(\R^4) \; .
$$
 \end{proposition}
 \begin{proposition}\label{prop4}
 Let $ T_0 $ : $ L^2\times \cdot\cdot \cdot \times L^2 \to L^1_{loc}(\R^3; \R) $ be a $ n$-linear operator. Assume that for some $ 2\le p,q\le \infty $,
 $$
 \| T_0(S(\cdot) \phi_1,\cdot\cdot\cdot, S(\cdot) \phi_n) \|_{L^p_t(\R; L^q(\R^3))} \lesssim 
 \prod_{i=1}^n \| \phi\|_{L^2} \; .
 $$
 Then there exists $ T\, :\, U^2_S \times \cdot\cdot \cdot \times U^2_S \to L^p_t(\R;L^q(\R^3)) $ satisfying 
  $$
 \| T(u_1,\cdot\cdot\cdot,  u_n) \|_{L^p_t(\R; L^q (\R^3))} \lesssim 
 \prod_{i=1}^n \| u_i\|_{U^2_S} 
 $$
 such that $T(u_1,\cdot\cdot\cdot, u_n)(t)(x,y,z)=T_0(u_1(t),\cdot\cdot\cdot, u_n(t))(x,y,z) $ almost everywhere. 
 \end{proposition}
 We are now ready  to define our resolution spaces : we denote by $ Y^{1,1}$ the space of all functions $ u\in {\mathcal S}'(\R^4) $ such that 
 $$
 \|u\|_{Y^{1,1}} := \sum_{N} N \| P_{N}u\|_{U^2_S} <\infty 
$$
and by $Y^{s,2} $ the space of all functions $ u\in {\mathcal S}'(\R^4) $ such that 
 $$
 \|u\|_{Y^{s,2}} := \Bigl( \sum_{N} N^{2s} \| P_{N}u\|_{U^2_S}^2 \Bigr)^{1/2} <\infty \; .
 $$
 Here, the Fourier projector $ P_N $ is the $ \R^3 $-version of the one defined in \eqref{proj}, i.e. $P_1 $ localized in frequencies 
  $ 3\xi^2+\mu^2+\eta^2 \lesssim 1 $ while  for  $ N\ge 2$, $ P_N $ localized in frequencies  $ 3\xi^2+\mu^2+\eta^2 \sim N $. 
  \subsubsection{Local existence estimate}
 Note that Proposition \ref{prop3} ensures that 
 \begin{equation}\label{65}
 \|\psi(\cdot) S(\cdot) u_0\|_{Y^{1,1}} \lesssim \|u_0\|_{B^{1,1}_2}  , \quad \forall u_0\in B^{1,1}_2(\R^3),  
 \end{equation}
 and 
 \begin{equation}
 \|\psi(\cdot) S(\cdot) u_0\|_{Y^{s,2}} \lesssim \|u_0\|_{H^s} , \quad \forall u_0\in H^s(\R^3) \; .
 \end{equation}
 Moreover, Proposition \ref{prop4} lead to the following estimates in $ U^2_S $ :
 \begin{lemma}\label{lem5}
 Let $ \psi \in C^\infty_c(\R) $. For any $ u\in U^2_S $ it holds 
 $$
 \|\psi u \|_{L^4} \lesssim \|u\|_{U^2_S}  \; .
 $$
 For any couple $ u,v \in  U^2_S $ and any couple $(N_1,N_2) $ of dyadic number such that $ N_1\ge 4 N_2 $ it holds 
 $$
 \|\psi P_{N_1} u P_{N_2} v \|_{L^2} \lesssim \frac{N_2}{N_1}  \|P_{N_1} u\|_{U^2_S}  \|P_{N_2} v\|_{U^2_S} \; .
 $$
 \end{lemma}
 \begin{proof}
The first estimate is a direct combination of the Strichartz estimate for the ZK equation in $ \R^3 $ (see \cite{LS})
$$
\| \psi S(\cdot) g \|_{L^4(\R^4)} \lesssim \| g\|_{L^2(\R^3)}
$$
with Proposition \ref{prop4}. 
To prove the second estimate we notice that since 
\begin{equation*}
\begin{split}
\Big| \frac{\partial H}{\partial \xi_1}&(\xi_1,\xi-\xi_1,\mu_1,\mu-\mu_1, \eta_1, (\eta-\eta_1))\Big| \\  =& 
\big|3\xi_1^2+\mu_1^2+\eta_1^2-(3(\xi-\xi_1)^2+(\mu-\mu_1)^2+(\eta-\eta_1)^2) \big| 
  \gtrsim  (N_1 \vee N_2)^2 \  .
\end{split}
\end{equation*}
where $ H $ is the resonance function in dimension 3, 
the $ \R^3 $-version of the bilinear estimate \eqref{BilinStrichartzI1} reads
$$
\|(P_{N_1}Q_{L_1}u_1) (P_{N_2}Q_{L_2}u_2)\|_{L^2}   \lesssim  
\frac{(N_1 \wedge N_2)}{N_1 \vee N_2}(L_1 \vee
L_2)^{\frac12}(L_1 \wedge L_2)^{\frac12}\|P_{N_1}Q_{L_1}u_1\|_{L^2}
\|P_{N_2}Q_{L_2}u_2\|_{L^2} \; .
$$
Since for $ \psi\in C^\infty_c(\R) $, $ g \in L^2(\R^3) $ and any dyadic number $ L \ge 1 $  it holds
$$
\|Q_L \psi S(\cdot) g \|_{L^2} \lesssim L^{-4} \|g \|_{L^2} 
$$
this ensures that 
$$
\|(\psi P_{N_1} S(\cdot) g)( \psi P_{N_2} S(\cdot) f) \|_{L^2} \lesssim \frac{N_2}{N_1}\|g\|_{L^2} \|f\|_{L^2} \; .
$$
The desired estimate follows by applying Proposition \ref{prop4}.
 \end{proof}
 We are now in position to prove the needed estimates on the retarded Duhamel operator.
 \begin{proposition}\label{prop6}
 For all $ u,v\in Y^{1,1} $ with compact support in time in $ ]-T,T[ $ it holds 
 \begin{equation}\label{po1}
 \Bigl\|\psi(t)\int_0^t S(t-t') \partial_x(u v)(t') \, dt' \Bigr\|_{Y^{1,1}} \lesssim T^{1/2} \|u\|_{Y^{1,1}} \|v\|_{Y^{1,1}} \; .
 \end{equation}
 For all  $u,v\in Y^{s,2}$, $s>1$,  with compact support in time in $ ]-T,T[ $ it holds 
  \begin{equation}\label{po2}
 \Bigl\|\psi(t)\int_0^t S(t-t') \partial_x(u v)(t') \, dt' \Bigr\|_{Y^{s,2}}\lesssim T^{1/2}\Bigl(  \|u\|_{Y^{1,1}} \|v\|_{Y^{s,2}}
 + \|u\|_{Y^{s,2}} \|v\|_{Y^{1,1}}\Bigr) \; .
 \end{equation}
 \end{proposition}
 \begin{proof}
 We separate the contribution of $ \sum_{N_1\not\sim N_2} P_{N_1} u P_{N_2} v $ and the one of $ \sum_{N_1\sim N_2} P_{N_1} u P_{N_2} v $.
 We use Proposition \ref{prop3}, Lemma \ref{lem5} and \eqref{zzz}. 
  For the first one we assume without loss of generality that $ N_1\ge 4 N_2 $ to  get 
   \arraycolsep1pt
  \begin{eqnarray*}
  \sum_{N} \sum_{N_1\ge 4 N_2} N\Bigl\|\psi(t)&&\int_0^t S(t-t') \partial_x P_N (u v)(t') \, dt' \Bigr\|_{U^2_S}\\
  && \lesssim   \sup_{\|w\||_{V^2_S}=1}\Bigl(  \sum_{N_1\ge 4  N_2} N_1 \| \partial_x (P_{N_1} u P_{N_2} v ) \|_{L^2} \|P_{N_1} w \|_{L^2} \Bigr)\\
  && \lesssim  T^{1/2}   \sup_{\|w\||_{V^2_S}=1}\Bigl(  \sum_{N_1\ge 4  N_2} N_1^2 \frac{N_2}{N_1} \|P_{N_1} u \|_{U^2_S} \|P_{N_2} v \|_{U^2_S}
  \|P_{N_1} w\|_{V^2_S} \Bigr) \\
  && \lesssim  T^{1/2} \|u\|_{Y^{1,1}}  \|v\|_{Y^{1,1}}  \; .
  \end{eqnarray*}
  Whereas the contribution of the second one is easily estimated by 
    \begin{eqnarray*}
  \sum_{N} \sum_{N_1\sim N_2\gtrsim N } N \Bigl\|\psi(t)&&\int_0^t S(t-t') \partial_x P_N (u v)(t') \, dt' \Bigr\|_{U^2_S}\\
  & & \lesssim   \sup_{\|w\||_{V^2_S}=1}\Bigl( \sum_{N}
  N^2 \sum_{N_1\sim N_2\gtrsim N} \|P_{N_1} u \|_{U^2_S} \|P_{N_2} v \|_{U^2_S} \|P_{N} w \|_{L^2} \Bigr)\\
  && \lesssim  T^{1/2}  \sum_{N_1} N_1^{-2} \sum_{l\ge 0} 2^{-l} N_1^2  (N_1\|P_{N_1} u \|_{U^2_S} )(N_1\|P_{N_1} v \|_{U^2_S})\\
    & &\lesssim T^{1/2} \|u\|_{Y^{1,1}}  \|v\|_{Y^{1,1}} \; .
  \end{eqnarray*}
  \arraycolsep5pt
  Finally the proof of \eqref{po2} follows the same lines and thus will be omitted.
 \end{proof}
 Note that the definition of the function space $ U^2_S $   ensures that for any $ 0<T<1 $ and any smooth function $ \psi\in C^\infty_c(\R) $ it holds
 $$
 \|\psi(\cdot/T) u\|_{U^2_S} \lesssim \|u\|_{U^2_S} , \quad \forall u\in U^2_S  \; .
 $$
  Therefore, combining \eqref{65} and  Proposition \ref{prop6}, we deduce that for any $ 0<T<1 $, the functional  
  $$
  {\mathcal G}_T(w)(t,\cdot):=\psi(t) S(t) u_0 -\frac{1}{2} \int_0^t S(t-t') \partial_x (\psi(\cdot/T) w)^2(t',\cdot)\, dt' 
  $$
  maps $Y^{1,1}$ into itself and satisfies
  $$
  \|{\mathcal G}_T(w) \|_{Y^{1,1}} \lesssim \|u_0\|_{B^{1,1}_2} + T^{1/2} \|w \|_{Y^{1,1}}^2 \; .
  $$
The Banach fixed point theorem then ensures  that for $ T\lesssim \|u_0\|_{B^{1,1}_2}^{-2} $, $ {\mathcal G}_T$ has got a fixed point  $ u $ satisfying 
 $ \|u \|_{Y^{1,1}}\le  2 \|u_0\|_{B^{1,1}_2}$. This proves the local existence and uniqueness  in the time  restriction space $ Y^{1,1} _T$ of the solution $ u $ of ZK emanating from 
 $ u_0\in B^{1,1}_2 (\R^3) $ 
  with a doubling time satisfying \eqref{Z3}. The result for $ u_0\in H^s(\R^3) $, $ s>1$, follows by noticing that \eqref{po2} implies that 
  $ {\mathcal G}_T $ maps as well $Y^{s,2}$ into itself  with 
  $$
  \|{\mathcal G}_T(w) \|_{Y^{s,2}} \lesssim \|u_0\|_{H^s} + T^{1/2}  \|w \|_{Y^{1,1}}\|w \|_{Y^{s,2}} \; .
  $$
  This completes the proof of Theorem \ref{prop2}.
   \begin{merci}
The authors were partially  supported by the Brazilian-French
program in mathematics. This work was initiated during a visit of the first author
at the Mathematical Institute of the Federal University of Rio de Janeiro (IM-UFRJ) in Brazil. He would like to thank the
IM-UFRJ for the kind hospitality\end{merci}

\end{document}